\documentclass[11pt]{article} 

\usepackage{amsmath}
\usepackage{lmodern}
\usepackage[T1]{fontenc}
\usepackage[babel=true]{microtype}
\usepackage{amsxtra}%
\usepackage{amsfonts}%
\usepackage{amssymb}%
\usepackage{amsthm}
\usepackage{graphicx}
\usepackage{epstopdf}
\usepackage{color,hyperref}
\hypersetup{colorlinks,breaklinks,
             linkcolor=blue,urlcolor=blue,
           anchorcolor=blue,citecolor=blue}
       
\usepackage{subfigure}
\usepackage{appendix}
\usepackage{enumerate}

\usepackage[margin=2.54cm]{geometry}

\newcommand{\eps}{\varepsilon}

\newcommand{\R}{\mathbb{R}}

\newcommand{\C}{\mathbb{C}}

\renewcommand{\phi}{\varphi}

\newcommand{\mcl}{\mathcal{L}}

\renewcommand{\Re}{\text{Re }}

\def\XXint#1#2#3{{\setbox0=\hbox{$#1{#2#3}{\int}$ }
		\vcenter{\hbox{$#2#3$ }}\kern-.6\wd0}}

\newtheorem{thm}{Theorem}
\newtheorem{prop}{Proposition}
\newtheorem*{thm*}{Theorem}

\newtheorem{lemma}[prop]{Lemma}

\newtheorem{hyp}{Hypothesis}

\numberwithin{equation}{section}
\numberwithin{prop}{section}

\title{Asymptotic stability of critical pulled fronts via resolvent expansions near the essential spectrum}
\author{Montie Avery \and Arnd Scheel}

\begin{document}
	
\maketitle

\begin{abstract}
	\noindent We study nonlinear stability of pulled fronts in scalar parabolic equations on the real line of arbitrary order, under conceptual assumptions on existence and spectral stability of fronts. In this general setting, we establish sharp algebraic decay rates and temporal asymptotics of perturbations to the front. Some of these results are known for the specific example of the Fisher-KPP equation, and our results can thus be viewed as establishing universality of some aspects of this simple model. We also give a precise description of how the spatial localization of perturbations to the front affects the temporal decay rate, across the full range of localizations for which asymptotic stability holds. Technically, our approach is based on a detailed study of the resolvent operator for the linearized problem, through which we obtain sharp linear time decay estimates that allow for a direct nonlinear analysis. 
\end{abstract}

\section{Introduction}

\subsection{Background and main results}

The formation of structure in spatially extended systems is often mediated by an invasion process, in which a pointwise stable state spreads into a pointwise unstable state. The Fisher-KPP equation 
\begin{equation}
u_t = u_{xx} + u - u^2 \label{e: Fisher}
\end{equation}
is a fundamental model for invasion processes, and much is known about invasion fronts in the Fisher-KPP equation. For all speeds $c \geq 2$, this equation has monotone traveling fronts $u(x,t) = q_c(x-ct)$ connecting the stable state $1$ to the unstable state $0$. The front with the minimum of these speeds, $c = 2$, which we call the \textit{critical front}, is distinguished for several reasons. Using comparison principles \cite{Kolmogorov, Comparison1, Lau, Aronson} or probabilistic methods relying on the relationship between the Fisher-KPP equation and branched Brownian motion \cite{Bramson1, Bramson2}, one may show that compactly supported initial conditions to \eqref{e: Fisher} spread with asymptotic speed 2. On the other hand, from the point of view of local stability, studying the critical front poses the greatest challenge. The stability of the supercritical fronts, with $c > 2$,  was first established by Sattinger \cite{Sattinger}, using exponential weights to move the essential spectrum to the left half plane. This is not possible for the critical front, due to the presence of absolute spectrum  \cite{AbsoluteSpecArndBjorn} at the origin for the linearization about the front -- with the optimal choice of weight, the essential spectrum is marginally stable, touching the imaginary axis at the origin. 

Stability of the critical front in \eqref{e: Fisher} was established by Kirchg\"assner \cite{Kirchgassner} and later refined using energy methods \cite{EckmannWayne}, renormalization group theory \cite{Bricmont, Gallay}, and most recently pointwise semigroup methods \cite{FayeHolzer}. While some of these papers consider equations of a more general form than \eqref{e: Fisher}, all are concerned with only second order, scalar (but possibly complex-valued) parabolic equations. From the point of view of time decay rates, the sharpest of these results is \cite{Gallay}, in which Gallay showed that sufficiently localization perturbations of the critical Fisher-KPP front decay with algebraic rate $t^{-3/2}$ and obtained a description of the leading order asymptotics of the solution for large time. The $t^{-3/2}$ decay rate was recently reobtained by Faye and Holzer \cite{FayeHolzer} using more direct pointwise semigroup methods, but without an asymptotic description of the solution. 

Here we study more general classes of equations. The main contributions of this paper are as follows:
\begin{enumerate}[(i)]
	\item We demonstrate that sharp nonlinear stability results on critical fronts depend only on conceptual assumptions on the existence and spectral stability of fronts, and not on the precise form of the equation considered. For instance, our results apply to equations without maximum principles. 
	\item We develop a new approach to the stability of critical fronts based on detailed estimates of the resolvent operator of the linearization near the branch point in the dispersion relation, which allow us to integrate along the essential spectrum when constructing the semigroup generated by the linearization. 
	\item We explore precisely how the spatial localization of perturbations to a critical front determines the algebraic time decay rate. 
\end{enumerate}

With a view towards pattern-forming systems which lack comparison principles in mind, we consider semilinear parabolic equations on the real line of arbitrary order of the form
\begin{align}
u_t = \mathcal{P} (\partial_x)u + f(u), \qquad u = u(x,t) \in \R, \,  t > 0, \,  x \in \R,  \label{e: eqn} 
\end{align} 
where $f$ is smooth, and $\mathcal{P}$ is a polynomial of the form
\begin{equation}
\mathcal{P} (\nu) = \sum_{k = 0}^{2m} p_k \nu^k, \quad (-1)^m p_{2m} < 0, \quad p_0 = 0.
\end{equation}
Hence $\mathcal{P}(\partial_x)$ is an elliptic operator of order $2m$. A key example is the fourth order extended Fisher-KPP equation, which can be derived as an amplitude equation near certain co-dimension 2 bifurcations in reaction-diffusion systems \cite{RottschaferDoelman}. Sixth order equations arise in the context of Rayleigh instabilities in fluid mechanics \cite[Section 3.3]{vanSaarloos} as well as in the phase field crystal model for elasticity and phase transitions \cite{Elder, ElderGalenko}. See the remarks in Section \ref{ss: remarks} on applicability of our methods to more general equations, and see Section \ref{s: discussion} for a discussion of several models to which our results directly apply.

We assume $f$ is smooth, with $f(0) = f(1) = 0$, $f'(0) > 0$, and $f'(1) < 0$. We are interested in invasion fronts connecting $u \equiv 1$ to $u \equiv 0$, and so we begin by discussing stability properties of these rest states for the full PDE \eqref{e: eqn} in a co-moving frame with speed $c$. The linearization about $u \equiv 0$ is then
\begin{align}
u_t = \mathcal{P}(\partial_x)u + c u_x + f'(0) u. \label{e: 0 linearization}
\end{align}
The $L^2$-spectrum of the constant-coefficient operator $\mathcal{P}(\partial_x) + c \partial_x + f'(0)$ is given, via the Fourier transform, by 
\begin{align}
\Sigma^+ = \{ \lambda \in \C: d_c^+ (\lambda, ik) = 0 \text{ for some } k \in \R \}.
\end{align}
where $d_c^+$ is the dispersion relation 
\begin{align}
d^+_c (\lambda, \nu) = \mathcal{P}(\nu) + c \nu + f'(0) - \lambda. 
\end{align}
A crucial feature of the Fisher-KPP front which we wish to retain is that the critical Fisher-KPP front is \textit{pulled}: it travels with the \textit{linear spreading speed}, i.e. the speed $c$ which marks the transition from pointwise growth to pointwise decay of compactly supported initial conditions to \eqref{e: 0 linearization}. Often these growth transitions are assumed to be captured by the presence of pinched double roots of the dispersion relation. We assume in the following hypothesis that there is a critical speed for which our dispersion relation has a \textit{simple} pinched double root at $\lambda = 0, \nu = -\eta_*$, which guarantees that this speed marks a transition from pointwise growth to pointwise decay. See \cite{HolzerScheelPointwiseGrowth} for a thorough description of linear spreading speeds and their relationship to pinched double roots. 

\begin{hyp}[Invasion at linear spreading speed] \label{hyp: spreading speed}
	We assume there exists a speed $c_*$ and an exponential rate $\eta_* > 0$ such that 
	\begin{enumerate}[(i)]
	\item (Simple pinched double root) For $\nu, \lambda$ near 0, we have
	\begin{align}
	d^+_{c_*} (\lambda, \nu - \eta_*) = \alpha \nu^2 - \lambda + \mathrm{O} (\nu^3) \label{e: right dispersion curve}
	\end{align}
	with $\alpha > 0$. 
	\item (Minimal critical spectrum) If $d^+_{c_*} (i \kappa, ik - \eta_*) = 0$ for some $k, \kappa \in \R$, then $k = \kappa = 0$. 
	\item (No unstable essential spectrum) $d^+_{c_*} (\lambda, i k - \eta_*) \neq 0$ for any $k \in \R$ and any $\lambda \in \C$ with $\Re \lambda > 0$. 
	\end{enumerate}
\end{hyp}
We refer to $c_*$ as the linear spreading speed, and from now on we fix $c = c_*$ and write $d^+_{c_*} = d^+$. One expects that the dynamics of pulled fronts are governed by the linearization at $u \equiv 0$, so we assume that the spectrum of the left rest state $u \equiv 1$ is stable in a strong sense, so that it does not interfere with the behavior on the right. The spectrum of the linearization about $u \equiv 1$, in the co-moving frame with speed $c_*$, is given by 
\begin{align}
\Sigma^- = \{ \lambda \in \C : d^- (\lambda, ik) = 0 \text{ for some } k \in \R \}, \label{e: left dispersion curve}
\end{align}
where $d^-$ is the left dispersion relation 
\begin{align}
d^- (\lambda, \nu) = \mathcal{P}(\nu) + c_* \nu + f'(1) - \lambda. 
\end{align}

\begin{hyp}[Stability on the left]\label{hyp: left}
	We assume that $\Re (\Sigma^-) < 0$. 
\end{hyp}

Front solutions $u(x,t) = q(x-c_*t)$ traveling with the linear spreading speed solve the traveling wave equation
\begin{align}
0 = \mathcal{P}(\partial_{\xi}) q + c_* \partial_\xi q + f(q), \label{e: TW ODE}
\end{align}
where $\xi = x - ct$.

\begin{hyp}[Existence of a critical front] \label{hyp: existence}
	We assume that \eqref{e: TW ODE} has a bounded solution $q_*$ with $q_*(\xi) \to 0$ as $\xi \to \infty$ and $q_*(\xi) \to 1$ as $\xi \to -\infty$, which we refer to as a critical front. 
\end{hyp}

The critical front $q_*$ is an equilibrium solution to \eqref{e: eqn} in a co-moving frame with speed $c_*$. Perturbations $v = u -q_*$ to a critical front $q_*$ solve 
\begin{align}
v_t = \mathcal{A} v + f(q_*+v) - f(q_*) - f'(q_*) v, \label{e: v eqn}
\end{align}
where $\mathcal{A}: H^{2m} (\R) \subseteq L^2 (\R) \to L^2 (\R)$ is the linearization about the front, 
\begin{align}
\mathcal{A} = \mathcal{P}(\partial_x) + c \partial_x + f'(q_*)
\end{align}
The assumption $f'(0)>0$ implies that the spectrum of $\mathcal{A}$ in $L^2$ is unstable, but Hypothesis \ref{hyp: spreading speed} guarantees that the essential spectrum of $\mcl = \omega \mathcal{A} \omega^{-1}$ is marginally stable, where $\omega$ is a smooth positive weight function satisfying 
\begin{align}
\omega(x) = \begin{cases}
e^{\eta_* x}, & x \geq 1, \\
1, & x \leq -1;
\end{cases} \label{e: omega def}
\end{align}
see Section \ref{ss: remarks} for details. In the Fisher-KPP equation, one has weak exponential decay of the critical front, $q_* (x) \sim x e^{-\eta_* x}$, and thus the derivative of the front does not give rise to a bounded solution to $\mcl u = 0$. We refer to the potential existence of such an $L^\infty$-eigenfunction as a \textit{resonance} at $\lambda = 0$. The lack of a resonance at $\lambda = 0$ for the Fisher-KPP linearization has been identified as an explanation for the faster $t^{-3/2}$ decay rate compared to the diffusive decay rate $t^{-1/2}$ \cite{ArndBjornEvansfcn}. Our analysis makes this observation precise, relying explicitly on the lack of a resonance at $\lambda = 0$. 
\begin{hyp}[No resonance or unstable point spectrum]\label{hyp: resonances}
	We assume that $\mcl : H^{2m} (\R) \subseteq L^2 (\R) \to L^2(\R)$ has no eigenvalues with $\Re \lambda \geq 0$. We additionally make the stronger assumption that there is no bounded pointwise solution to $\mcl u = 0$. 
\end{hyp}

We introduce algebraic weights to manage further subtleties in the localization of perturbations. For $r_{\pm} \in \R$, we define a smooth positive weight function $\rho_{r_-,r_+}$ satisfying 
\begin{align}
\rho_{r_-, r_+} (x) = \begin{cases}
\langle x \rangle^{r_+}, & x \geq 1, \\
\langle x \rangle^{r_-}, & x \leq -1,
\end{cases}
\end{align}
where $\langle x \rangle = (1+x^2)^{1/2}$. Using these weights, we define algebraically weighted Sobolev spaces $H^k_{r_-,r_+} (\R)$ through the norms 
\begin{align}
\|g\|_{H^k_{r_-,r_+}} = \| \rho_{r_-,r_+} g \|_{H^k}. 
\end{align}
For $k = 0$, we write $H^0_{r_-,r_+} (\R) = L^2_{r_-,r_+} (\R)$. If $r_- = 0$, $r_+ = r$, we write $\rho_r = \rho_{0, r}$ and denote the corresponding function space by $H^k_r (\R)$. 

We are now ready to state our main results. First, we show that the sharp decay rate $t^{-3/2}$ for sufficiently localized perturbations obtained by Gallay \cite{Gallay} and Faye and Holzer \cite{FayeHolzer} for the Fisher-KPP equation is valid in this general setting. Even in the Fisher-KPP setting, our result refines that of \cite{FayeHolzer} in the sense that Faye and Holzer require some exponential localization of perturbations on the left as well as on the right, which we show is not necessary. 

\begin{thm}[Stability with sharp decay rate]\label{t: stability}
	Assume Hypotheses \ref{hyp: spreading speed} through \ref{hyp: resonances} hold, and fix $r > 3/2$. There exist constants $\eps > 0$ and $C > 0$ such that if $\|\omega v_0\|_{H^1_r} < \eps$, then 
	\begin{align}
	\|\omega(\cdot) v(\cdot, t)\|_{H^1_{-r}} \leq \frac{C}{(1+t)^{3/2}} \|\omega v_0\|_{H^1_r}. 
	\end{align}
\end{thm}

Next, for more strongly localized data, we obtain an asymptotic description of the solution profile for large times, recovering Gallay's result \cite{Gallay} for the Fisher-KPP equation based on renormalization group theory. 

\begin{thm}[Stability with asymptotics]\label{t: asymptotics}
	Assume Hypotheses \ref{hyp: spreading speed} through \ref{hyp: resonances} hold, and let $\psi \in H^{2m}_s (\R), s < - \frac{3}{2},$ be the (unique up to a constant multiple) solution to $\mcl \psi = 0$ which is linearly growing at $+\infty$ and exponentially localized on the left. For any fixed $r > \frac{5}{2}$, there exist constants $\eps > 0$ and $C > 0$ such that if $\|\omega v_0\|_{H^1_r} < \eps$, then there is a real number $\alpha_* = \alpha_* (\omega v_0)$, depending smoothly on $\omega v_0$ in $H^1_r (\R)$ such that for $t > 1$, 
	\begin{align*}
	\|\omega(\cdot) v(\cdot, t) - \alpha_* t^{-3/2} \psi(\cdot)\|_{H^1_{-r}} \leq \frac{C}{(1+t)^2} \|\omega v_0\|_{H^1_r}.  
	\end{align*} 
\end{thm}

Our methods are based on studying the regularity of the resolvent $(\mcl - \lambda)^{-1}$ in $\gamma = \sqrt{\lambda}$, with a suitable branch cut. In the setting of Theorem \ref{t: stability}, we show that the resolvent is Lipschitz in $\gamma$ near the origin in an appropriate sense. With more localization, we expand the resolvent to higher order, which allows us to identify the leading order asymptotics of the semigroup $e^{\mcl t}$ used to prove Theorem \ref{t: asymptotics}. At lower levels of localization, the resolvent loses Lipschitz continuity but first retains some H\"older continuity. As we allow for even less localized perturbations, the resolvent blows up near the origin, but with a quantifiable rate. In these respective settings, we obtain the following two theorems, giving a precise description of the relationship between spatial localization of their perturbations and their algebraic decay rates, which appears to be new even in the setting of the Fisher-KPP equation. 

\begin{thm}[Stability -- moderate localization]\label{t: stability resolvent Holder}
	Assume Hypotheses \ref{hyp: spreading speed} through \ref{hyp: resonances} hold. Fix $\frac{1}{2} < r < \frac{3}{2}$ and $s < r - 2$. For any $0 < \alpha < r - \frac{3}{2} + \min \left(1, -\frac{1}{2} - s \right)$, there exist positive constants $C$ and $\eps$ such that if $\|\omega v_0\|_{H^1_r} < \eps$, then 
	\begin{align}
	\|\omega(\cdot) v(\cdot, t)\|_{H^1_s} \leq \frac{C}{(1+t)^{1 + \frac{\alpha}{2}}} \|\omega v_0\|_{H^1_r}.  
	\end{align}
\end{thm}

\begin{thm}[Stability -- minimal localization]\label{t: stability resolvent blowup}
	Assume Hypotheses \ref{hyp: spreading speed} through \ref{hyp: resonances} hold. Fix $-\frac{3}{2} < r < 1/2$ and $s < r -2$. For any $\frac{1}{2} - r < \beta < -s - \frac{3}{2}$, there exist positive constants $C$ and $\eps$ such that if $\|\omega v_0\|_{H^1_r} < \eps$, 
	then  
	\begin{align}
	\|\omega(\cdot) v(\cdot, t)\|_{H^1_s} \leq \frac{C}{(1+t)^{1-\frac{\beta}{2}}} \|\omega v_0\|_{H^1_r}. 
	\end{align}
\end{thm}

Note, choosing $r \gtrsim - \frac{3}{2}$ and $s \lesssim - \frac{7}{2}$, the optimal choice for $\beta$ is $\beta \lesssim 2$, thereby giving arbitrarily slow algebraic decay. For the remainder of the paper, we assume Hypothesis \ref{hyp: spreading speed} through \ref{hyp: resonances} hold. 

\subsection{Preliminaries, notation, and remarks}\label{ss: remarks} \hfill

\noindent \textbf{General exponential weights.} In our analysis of the resolvent, we will use exponential weights on the right to move the essential spectrum of $\mcl$ in order to regain Fredholm properties at the origin. Given $\eta \in \R$, we let $\omega_\eta$ be a smooth positive weight function satisfying 
\begin{align}
\omega_{\eta} (x) = \begin{cases}
e^{\eta x}, & x \geq 1, \\
1, & x \leq -1. 
\end{cases}
\end{align}
Given a non-negative integer $k$, we define the exponentially weighted Sobolev space $H^k_{\mathrm{exp}, \eta} (\R)$ through the norm 
\begin{align}
||g||_{H^k_{\mathrm{exp}, \eta}} = ||\omega_{\eta} g||_{H^k}. 
\end{align}
If $k = 0$, we write $H^0_{\mathrm{exp}, \eta} (\R) = L^2_{\mathrm{exp}, \eta} (\R)$.

\begin{figure}
	\centering
	\includegraphics[width = .9\textwidth]{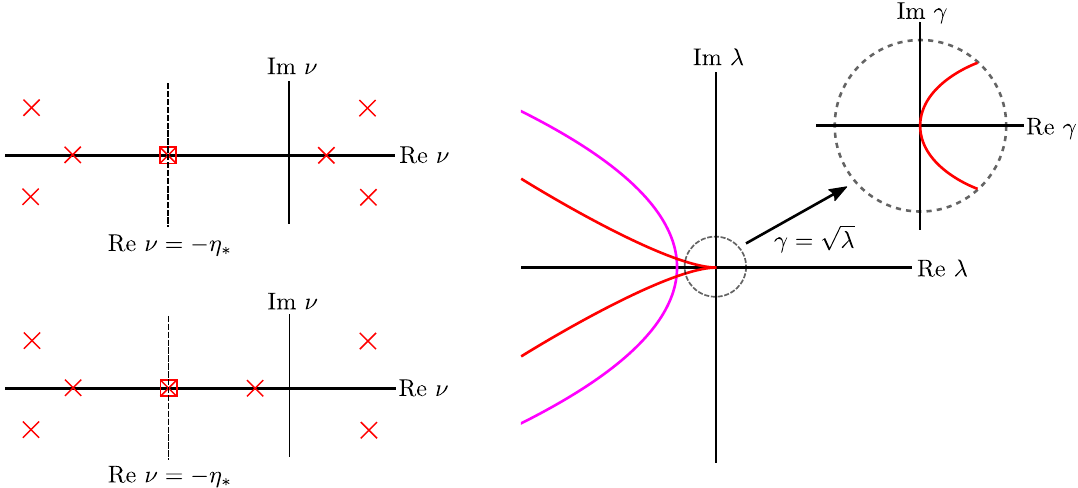}
	\caption{Left: the two possibilities for the location of the spatial eigenvalues $\nu$ of the asymptotic system at $+\infty$ for $\lambda = 0$, according to Hypothesis \ref{hyp: spreading speed}. The red square around the spatial eigenvalue at $\nu = -\eta_*$ indicates the presence of a Jordan block there. Right: Fredholm borders of $\mcl$ associated to $+\infty$ (red) and $-\infty$ (magenta); the inset shows the image of a neighborhood of the origin under the map $\gamma = \sqrt{\lambda}$.  }
	\label{fig: spatial evals and spec}
\end{figure} 

\noindent \textbf{Spectrum of the linearization.} We say $\lambda \in \C$ is in the essential spectrum of an operator $A$ if $A - \lambda$ is not an index zero Fredholm operator. The assumptions that $f'(0) > 0$ and $f'(1) < 0$ imply that the critical front $q_*$ converges to its limits exponentially quickly, so the coefficients of $\mathcal{A}$ attain limits exponentially quickly as $x \to \pm \infty$. By Palmer's theorem, the essential spectrum of $\mathcal{A}$ is determined by the asymptotic dispersion relations. The dispersion curves $\Sigma^\pm$, given in \eqref{e: right dispersion curve} and \eqref{e: left dispersion curve}, are the \textit{Fredholm borders} of $\mathcal{A}$: $\mathcal{A} - \lambda$ is Fredholm if and only if $\lambda \notin \Sigma^+ \cup \Sigma^-$. Due to well-posedness of the underlying PDE, this implies that $A - \lambda$ is Fredholm index zero if $\lambda$ is to the right of $\Sigma^+ \cup \Sigma^-$, and hence the dispersion curves give a sharp upper estimate of the location of the essential spectrum. 

Locating the essential spectrum in an exponentially weighted space with weight $\omega_\eta$ is equivalent to studying the spectrum of the conjugate operator $\omega_\eta \mathcal{A} \omega^{-1}_\eta$ in $L^2$, since multiplication by $\omega_\eta$ is an isomorphism from $L^2_{\mathrm{exp}, \eta} (\R)$ to $L^2 (\R)$. Operators of this form still have exponentially asymptotic coefficients, but conjugation by the weight changes the limits at $\pm \infty$ and hence moves the essential spectrum. Using the exponential weight $\omega = \omega_{\eta_*}$ defined in \eqref{e: omega def}, the limiting operators at $\pm \infty$ are
\begin{align}
\mcl_+ &= \mathcal{P}(\partial_x - \eta_*) + c_* (\partial_x - \eta_*) + f'(0), \label{e: lplus def} \\
\mcl_- &= \mathcal{P}(\partial_x) + c_* \partial_x + f'(1). 
\end{align}
One finds that the right dispersion curve for $\mcl = \omega \mathcal{A} \omega^{-1}$ is 
\begin{align}
\Sigma^+_{\eta_*} = \{ \lambda \in \C: d^\pm_{c_*} (\lambda, \nu) = 0 \text{ for some } \nu \in \C \text{ with } \Re \nu = - \eta_*\}. 
\end{align}
Hypothesis \ref{hyp: spreading speed} then guarantees that this choice of $\eta_*$ pushes the essential spectrum as far left as possible (due to the presence of absolute spectrum \cite{AbsoluteSpecArndBjorn} at the origin), and that with this choice of weight, the spectrum of $\mcl$ touches the imaginary axis at the origin and nowhere else. See the right panel of Figure \ref{fig: spatial evals and spec} for a depiction of the Fredholm borders of $\mcl$, and see \cite{FiedlerScheel, KapitulaPromislow} for further details on the essential spectrum of operators of this type. 

\noindent \textbf{Spatial eigenvalues and asymptotics of the front.} When one writes the traveling wave equation \eqref{e: TW ODE} as a first order system with coordinates $Q = (q, q', ..., q^{(2m-1)})$ and linearizes about the equilibrium $Q = 0$, obtaining an equation $Q' = AQ$, Hypothesis \ref{hyp: spreading speed} implies that the matrix $A$ has a Jordan block of length two at $\nu = - \eta_*$ \cite{HolzerScheelPointwiseGrowth}. If there are no slower-decaying stable eigenvalues, that is, if
\begin{align}
- \eta_* = \max \{ \Re \nu : \nu \in \sigma (A) \text{ with } \Re \nu < 0 \}, \label{e: spatial eval gap}
\end{align}
then, counting the dimensions of stable and unstable manifolds, one expects that the critical front $q_*$, solving \eqref{e: TW ODE} with $c = c_*$, is locally unique up to translation invariance, and that it inherits the decay rate from the Jordan block, that is
\begin{align}
q_* (x) \sim x e^{-\eta_* x}, \quad x \to \infty. \label{e: front asymptotics infty}
\end{align}
This is the situation pictured in the top left panel of Figure \ref{fig: spatial evals and spec}. Since we are assuming $\mcl$ has no resonances, \eqref{e: front asymptotics infty} must hold in this case, since otherwise we would have $|q_*'(x)| \leq C e^{- \eta_* x}$ for $x$ large, which would imply that $\mcl$ has a resonance at $\lambda = 0$. 

On the other hand, if $A$ has another eigenvalue $\nu$ with $- \eta_* < \Re \nu < 0$, as pictured in the bottom left panel of Figure \ref{fig: spatial evals and spec}, then one expects that fronts with speed $c_*$ come in a two-parameter family, with one parameter arising from translation invariance. Typically these fronts decay exponentially as $x \to \infty$ but with a rate slower than $-\eta_*$. In this case, our results apply to any of these fronts in this two-parameter family. 

\noindent \textbf{Exponential expansions and uniqueness of the front.} Solutions to the equation $\mcl u =0$ have \textit{exponential expansions}, in the sense that solutions which are at most linearly growing at infinity have the form
\begin{align*}
u(x) = \chi_+ (x) (\mu_0 + \mu_1 x) + w(x),
\end{align*}
where $\chi_+$ is a smooth positive cutoff function satisfying 
\begin{align}
\chi_+ (x) = \begin{cases}
0, &x \leq 2 \\
1, &x \geq 3,
\end{cases} \label{e: chiplus}
\end{align}
and $w$ is exponentially localized. This decomposition follows from the presence of a Jordan block at the origin when writing $\mcl_+ u = 0$ as a first-order system, with the rest of the eigenvalues away from the imaginary axis. From this characterization, we conclude that there is a unique solution to $\mcl u = 0$ which is linearly growing at $+\infty$, up to a constant multiple: otherwise, a linear combination of two distinct solutions would give rise to a resonance at $\lambda = 0$. This justifies the claim of uniqueness of $\psi$ in the statement of Theorem \ref{t: asymptotics}. 

Furthermore, if \eqref{e: spatial eval gap} holds, then $\omega q_*'$ is linearly growing at $\infty$, by \eqref{e: front asymptotics infty}. Since $\mcl (\omega q_*') = 0$ by translation invariance of \eqref{e: eqn}, we conclude that in this case we have $\psi = \omega_{\eta_*} q_*'$ (fixing the constant multiple appropriately). 

\noindent \textbf{Threshold for asymptotic stability.} We note that Theorem \ref{t: stability resolvent blowup} is sharp in the sense that asymptotic stability is no longer true for initial data in $H^1_r(\R)$ with $r < -\frac{3}{2}$, and accordingly the algebraic decay rate in Theorem \ref{t: stability resolvent blowup} goes to zero as $r \to -\frac{3}{2}^+$. On the linear level, this can be seen from the fact that $\psi \in H^1_r (\R)$ for $r < -\frac{3}{2}$, and $e^{\mcl t} \psi = \psi$ since $\mcl \psi = 0$. On the nonlinear level, if \eqref{e: spatial eval gap} holds, then using the asymptotics \eqref{e: front asymptotics infty}, one sees that using a small shift of the critical front as an initial condition is a perturbation which is small in $H^1_r (\R)$ for $r < -\frac{3}{2}$. The shifted front is still an equilibrium solution, so asymptotic stability does not hold for the nonlinear equation. 

\noindent \textbf{More general equations.} Since we already control all derivatives up to order $2m-1$ in our linear decay estimates in Proposition \ref{p: linear time decay}, our results readily extend to the general semilinear case, where $f = f(u, u_x, ..., \partial_x^{2m-1} u)$. With mostly editorial modifications, our methods should also apply to systems of semilinear parabolic equations. We focus on the scalar case with $f = f(u)$ here for clarity of presentation. 

\noindent \textbf{Additional notation.} For two Banach spaces $X$ and $Y$, we let $\mathcal{B}(X,Y)$ denote the space of bounded linear operators from $X$ to $Y$, with the operator norm topology. For $\delta> 0$, we let $B(0, \delta)$ denote the ball centered at the origin in $\C$ with radius $\delta$. 

\noindent \textbf{Outline.} The remainder of the paper is organized as follows. We first focus on the necessary ingredients for the proofs of Theorems \ref{t: stability} and \ref{t: asymptotics}, to clearly demonstrate our approach for analyzing the resolvent. We start by analyzing the resolvent of the limiting operator $(\mcl_+ - \gamma^2)^{-1}$ in Section \ref{sec: asymptotic resolvents}, by obtaining pointwise estimates on the integral kernel for this resolvent. In Section \ref{s: full resolvent}, we then transfer our estimates to the full resolvent $(\mcl - \gamma^2)^{-1}$, by decomposing our data and solution into left, right, and center pieces, solving the left and right pieces with the asymptotic operators, and using a far-field/core decomposition as developed in \cite{PoganScheel} to solve the center piece. 

In Section \ref{s: semigroup estimates}, we construct the semigroup $e^{\mcl t}$ via a contour integral, and use our resolvent estimates to obtain sharp decay rates and an asymptotic expansion for large time for this semigroup through a careful choice of the integration contour. With these linear decay estimates in hand, we establish nonlinear stability in Section \ref{s: nonlinear stability} via a direct argument, proving Theorem \ref{t: stability} -- the principle challenge in this problem is in obtaining optimal linear estimates, rather than handling the nonlinearity. In Section \ref{s: nonlinear asymptotics}, we again use a direct argument to transfer large time asymptotics for the semigroup $e^{\mcl t}$ to asymptotics for the solution for the nonlinear equation, proving Theorem \ref{t: asymptotics}. In Section \ref{s: lower localization}, we describe the modifications necessary to handle less localized initial conditions, proving Theorems \ref{t: stability resolvent Holder} and \ref{t: stability resolvent blowup}. We conclude in Section \ref{s: discussion} by giving examples of systems to which our results apply and discussing some subtleties surrounding our assumptions. 

\noindent \textbf{Acknowledgements.} This material is based upon work supported by the National Science Foundation through the Graduate Research Fellowship Program under Grant No. 00074041, as well as through NSF-DMS-1907391. Any opinions,
findings, and conclusions or recommendations expressed in this material are those of the
authors and do not necessarily reflect the views of the National Science Foundation.

\section{Resolvents for asymptotic operators}\label{sec: asymptotic resolvents}
In this section, we establish regularity properties in $\lambda$ for the resolvents $(\mathcal{L}_\pm - \lambda)^{-1}$ of the limiting operators. Since the dispersion relation has a degree 2 branch point at the origin, roots of the dispersion relation are therefore analytic functions of $\gamma = \sqrt{\lambda}$ near $\gamma = 0$, and so we study regularity in $\gamma$ near this branch point. We choose the branch cut along the negative real axis, so that $\Re \gamma > 0$. We let $R^+ (\gamma) = (\mcl_+ - \gamma^2)^{-1}$. The key result of this section is the following proposition, which gives expansions for $R^+(\gamma)$ to finite order in $\gamma$, depending on the amount of algebraic localization required, when restricting to odd functions.  

\begin{prop}\label{p: right resolvent estimate}
	Let $r > 3/2$. There is a limiting operator $R^+_0$, which is a bounded operator from $L^2_{s,s} (\R)$ to $H^{2m-1}_{-r, -r} (\R)$ for any $s > \frac{1}{2}$, and a constant $C > 0$ such that for any odd function $g \in L^2_{r, r} (\R)$, we have 
	\begin{align}
	\|(R^+(\gamma)-R^+_0) g\|_{H^{2m-1}_{-r, -r}} \leq C |\gamma| \|g\|_{L^2_{r,r}} 
	\end{align}
	for all $\gamma$ sufficiently small with $\gamma^2$ to the right of $\Sigma^+_{\eta_*}$. 
	
	If $r > 5/2$, then in addition there is an operator $R^+_1 : L^2_{r, r} (\R) \to H^{2m-1}_{-r, -r} (\R)$ and a constant $C > 0$ such that for any odd function $g \in L^2_{r,r} (\R)$, we have
	\begin{align}
	\|(R^+(\gamma) - R^+_0 - \gamma R^+_1) g\|_{H^{2m-1}_{-r, -r}} \leq C |\gamma|^2 \|g\|_{L^2_{r,r}} \label{e: right second order expansion}
	\end{align}
	for all $\gamma$ sufficiently small with $\gamma^2$ to the right of $\Sigma^+_{\eta_*}$. 
\end{prop}

To prove this, we construct the Green's function for the resolvent equation via a reformulation as a first order system. Hypothesis \ref{hyp: spreading speed} will guarantee that the dynamics in this system are to leading order the same as for the system corresponding to the heat equation on the real line. Restricting to odd initial data then improves the regularity of the resolvent by introducing effective absorption into the system. Since the equation $(\mcl_+ - \gamma^2) u = g$ has constant coefficients, the solution operator is given by convolution with a Green's function $G^+_\gamma$, which solves 
\begin{align}
(\mcl_+ - \gamma^2) G^+_\gamma = - \delta_0,
\end{align}
where $\delta_0$ is the Dirac delta distribution supported at the origin. We now write $\mcl_+$ as 
\begin{align}
\mcl_+ = \sum_{k=2}^{2m} c_k \partial_x^{k}. 
\end{align}
As in \cite{HolzerScheelPointwiseGrowth}, we recast $(\mcl_+ - \gamma^2) u = g$ as a first-order system in $U = (u, \partial_x u, ..., \partial_x^{2m-1} u)$, and find
\begin{align}
\partial_x U = M(\gamma) U + F,
\end{align}
where $F = (0, 0, ..., 0, g)^T$, and $M(\gamma)$ is a $2m$-by-$2m$ matrix
\begin{equation}
M(\gamma) = \begin{bmatrix}
0 & 1 & 0 & \dots & 0 \\
0 & 0 & 1 & \dots & 0 \\
\vdots & \vdots & 0 & \hspace{-0.7cm} \ddots & \vdots \\
0 & \dots & 0 & 1 & 0 \\
0 & \dots & \dots & 0 & 1 \\
\gamma^2/c_{2m} & 0 & -c_2/c_{2m} & \dots & - c_{2m-1}/c_{2m} 
\end{bmatrix}. \label{e: matrix formula}
\end{equation}

By Palmer's theorem \cite{FiedlerScheel, KapitulaPromislow}, if $\gamma^2$ is to the right of the essential spectrum $\Sigma^+_{\eta_*}$, then $M(\gamma)$ is a hyperbolic matrix, with stable and unstable subspaces $E^{\mathrm{s/u}} (\gamma)$ satisfying $\dim E^s (\gamma) = \dim E^u (\gamma)$. We let $P^\mathrm{s} (\gamma)$ and $P^\mathrm{u} (\gamma) = I-P^\mathrm{s}(\gamma)$ denote the corresponding spectral projections onto these subspaces. The matrix Green's function $T_\gamma$ for this system solves 
\begin{align}
(\partial_x - M(\gamma)) T_\gamma = - \delta_0 I,
\end{align}
where $I$ is the identity matrix of size $2m$-by-$2m$. The matrix Green's function is given by 
\begin{align}
T_\gamma (x) = \begin{cases}
- e^{M(\gamma) x} P^\mathrm{s} (\gamma), & x > 0 \\
e^{M(\gamma) x} P^\mathrm{u} (\gamma), & x < 0. 
\end{cases}
\end{align}
The scalar Green's function $G_\gamma$ is recovered from $T_\gamma$ through 
\begin{align}
G_\gamma^+ = P_1 T_\gamma Q_1 c_{2m}^{-1}, \label{e: scalar Green fcn}
\end{align}
where $P_1$ is the projection onto the first component and $Q_1$ is the embedding into the last component, i.e. $P_1 (u_1, ..., u_{2m}) = u_1$ and $Q_1 g = (0, ..., 0, g)^T$. From these formulas, since $M(\gamma)$ is analytic in $\gamma^2$, we see that the only obstructions to regularity in $\gamma$ of $G_\gamma^+$ are singularities in the projections $P^{\mathrm{s/u}}(\gamma)$. Such a singularity does occur: the structure of $M(\gamma)$, arising from writing a scalar equation as a first-order system, implies that 
\begin{align}
\det (M(\gamma) - \nu) = d^+(\gamma^2, \nu-\eta_*). 
\end{align}
Hence the spatial eigenvalues $\nu$ of $M(\gamma)$ are roots of the dispersion relation, satisfying
\begin{align}
0 = d^+(\gamma^2, \nu - \eta_*) = \alpha \nu^2 - \gamma^2 + \mathrm{O}(\nu^3), 
\end{align}
with $\alpha > 0$. Solving near the origin with the Newton polygon, one finds two solutions bifurcating from the origin, given by 
\begin{align}
\nu^{\pm} (\gamma) = \pm\frac{1}{\sqrt{\alpha}} \gamma + \mathrm{O}(\gamma^2). \label{e: spatial evals}
\end{align}
As $\gamma$ approaches zero from the right of the essential spectrum, $\nu^{\pm}$ merge to form a $2$-by-$2$ Jordan block to the eigenvalue zero, necessarily giving rise to a singularity in $P^{\mathrm{s/u}} (\gamma)$ \cite{Kato}. With the Newton polygon, one readily finds that these are the only eigenvalues of $M(\gamma)$ near the origin for $\gamma$ small. 

We therefore isolate the singularity by splitting the projections as 
\begin{align}
P^\mathrm{s/u}(\gamma) = P^{\mathrm{cs/cu}} (\gamma) + P^{\mathrm{ss/uu}} (\gamma),
\end{align}
for $\gamma^2$ to the right of the essential spectrum, where $P^{\mathrm{cs/cu}} (\gamma)$ are the spectral projections onto the one-dimensional eigenspaces associated to $\nu^\pm(\gamma)$, respectively, and $P^{\mathrm{ss/uu}} (\gamma)$ are the spectral projections onto the rest of the stable/unstable eigenvalues, respectively. Standard spectral perturbation theory \cite{Kato} implies that $P^{\mathrm{ss/uu}} (\gamma)$ are analytic in $\gamma^2$ for $\gamma$ small. We characterize the singularities of $P^{\mathrm{cs/cu}} (\gamma)$ in the following lemma.

\begin{lemma}\label{l: projection pole}
	The projections $P^{\mathrm{cs/cu}} (\gamma)$ have poles of order 1 at $\gamma = 0$, with expansions
	\begin{align}
	P^{\mathrm{cs/cu}} (\gamma) = \pm \frac{1}{\gamma} P_{-1} + \mathrm{O}(1)
	\end{align}
	near $\gamma = 0$. In particular, the poles in these expansions differ only by a sign. Furthermore, the top right entry of $P_{-1}$ is nonzero. We denote the remainder term by
	\begin{align}
	\tilde{P}^{\mathrm{cs/cu}}(\gamma) = P^\mathrm{cs/cu}(\gamma) \mp \frac{1}{\gamma} P_{-1}. 
	\end{align}
\end{lemma} 
\begin{proof}
	Since for $\gamma$ nonzero $\nu^{\pm} (\gamma)$ are each algebraically simple eigenvalues of $M(\gamma)$, we can construct the projections onto their eigenspaces via Lagrange interpolation. This approach gives a formula sometimes known as the Frobenius covariant. We order the eigenvalues of $M(\gamma)$ as $(\nu_1(\gamma), \nu_2(\gamma), ..., \nu_{2m}(\gamma))$, repeating eigenvalues according to algebraic multiplicity if there are non-trivial Jordan blocks in the strong stable/unstable subspaces, with $\nu_1(\gamma) = \nu^+(\gamma)$ and $\nu_2(\gamma) = \nu^-(\gamma)$. The center stable projection is then given by \begin{align}
	P^\mathrm{cs}(\gamma) = \prod_{k = 1, k\neq 2}^{2m} \frac{1}{\nu^-(\gamma) - \nu_k(\gamma)} (M(\gamma) - \nu_k (\gamma) I).
	\end{align}
	Repeating the eigenvalues according to algebraic multiplicity guarantees that the right hand side annihilates all the other eigenspaces, and one can check that the normalization guarantees it gives the spectral projection. Since all the other eigenvalues are bounded away from zero for $\gamma$ small, the only singularity arises from the factor $(\nu^-(\gamma) - \nu^+(\gamma))^{-1}$. Using the fact that $\nu^-(\gamma) - \nu^+(\gamma) = - \frac{2}{\sqrt{\alpha}} \gamma + \mathrm{O}(\gamma^2)$, we write
	\begin{align*}
	\gamma P^\mathrm{cs}(\gamma) \big|_{\gamma = 0} = -\frac{\sqrt{\alpha}}{2} (M(0)-\nu^+(0) I) \prod_{k=3}^{2m} \frac{1}{-\nu_k(0)} (M(0)-\nu_k(0) I). 
	\end{align*}
	Note that this is a polynomial of degree $2m-1$ in $M(0)$. From the form of $M(\gamma)$ in \eqref{e: matrix formula}, one sees that the top right entry of $M(0)^{2m-1}$ is equal to 1, and the top right entry of $M(0)^{k}$ is zero for all $k < 2m-1$. Hence the top right entry of $\gamma P^\mathrm{cs}(\gamma) |_{\gamma = 0}$ is
	\begin{align}
	\beta := -\frac{\sqrt{\alpha}}{2} \prod_{k=3}^{2m} \left(-\frac{1}{\nu_k(0)}\right),  
	\end{align} 
	which is nonzero. Repeating the argument for 
	\begin{align*}
	P^\mathrm{cu} (\gamma) = \prod_{k = 2}^{2m} \frac{1}{\nu^+(\gamma) - \nu_k (\gamma)} (M(\gamma) - \nu_k (\gamma) I),
	\end{align*}
	one readily finds $\gamma P^\mathrm{cu}(\gamma) |_{\gamma = 0} = - \gamma P^{\mathrm{cs}} (\gamma)|_{\gamma = 0}$, completing the proof of the lemma. 
\end{proof}

We now use this result to expand the formula \eqref{e: scalar Green fcn} for $G_\gamma^+$. For $x \geq 0$, we have 
\begin{align*}
G_\gamma^+ (x) &= -P_1 e^{M(\gamma) x} (P^\mathrm{cs}(\gamma) + P^\mathrm{ss}(\gamma)) Q_1 c_{2m}^{-1} \\
  &= -c_{2m}^{-1} \frac{\beta}{\gamma} e^{\nu^-(\gamma) x} -c_{2m}^{-1} e^{\nu^- (\gamma) x} P_1 \tilde{P}^\mathrm{cs}(\gamma) Q_1  - c_{2m}^{-1} P_1 e^{M(\gamma) x} P^\mathrm{ss}(\gamma) Q_1,
\end{align*}
and for $x < 0$, we have 
\begin{align*}
G_\gamma^+ (x) = - c_{2m}^{-1} \frac{\beta}{\gamma} e^{\nu^+ (\gamma) x} + c_{2m}^{-1} e^{\nu^+ (\gamma) x} P_1 \tilde{P}^{\mathrm{cu}} (\gamma) Q_1 + c_{2m}^{-1} P_1 e^{M(\gamma) x} P^\mathrm{uu}(\gamma) Q_1. 
\end{align*}
The leading term is the only term which is singular in $\gamma$. Lemma \ref{l: projection pole} guarantees that this term has the same coefficient for $x \geq 0$ and $x < 0$. We now show that this term can be replaced by (essentially) the resolvent kernel for the heat equation, and that the remaining error terms can be controlled as well, so that the behavior is the same as for the resolvent in the heat equation. Let 
\begin{align}
G_\gamma^c (x) = \begin{cases}
-c_{2m}^{-1} \frac{\beta}{\gamma} e^{\nu^- (\gamma) x}, & x \geq 0 \\
-c_{2m}^{-1} \frac{\beta}{\gamma} e^{\nu^+ (\gamma) x}, & x < 0,
\end{cases}
\end{align}
and let 
\begin{align}
G_\gamma^{\mathrm{heat}} (x) = \begin{cases}
-c_{2m}^{-1} \frac{\beta}{\gamma} e^{-\nu_0 \gamma x}, & x \geq 0 \\
-c_{2m}^{-1} \frac{\beta}{\gamma} e^{\nu_0 \gamma x}, & x < 0,
\end{cases}
\end{align}
where $\nu_0 = \frac{1}{\sqrt{\alpha}}$. 
We separate the resolvent kernel into four pieces
\begin{align}
G_\gamma^+ = G_\gamma^\mathrm{heat} + (G^c_\gamma - G^\mathrm{heat}_\gamma)  + \tilde{G}_\gamma^c  + G^h_\gamma, \label{e: right greens fcn decomposition}
\end{align}
where $\tilde{G}^c_\gamma$ consists of the remainder term associated to the central spatial eigenvalues
\begin{align}
\tilde{G}_\gamma^c (x) = \begin{cases}
- c_{2m}^{-1} e^{\nu^- (\gamma) x} P_1 \tilde{P}^\mathrm{cs}(\gamma) Q_1, & x \geq 0 \\
c_{2m}^{-1} e^{\nu^+(\gamma) x} P_1 \tilde{P}^\mathrm{cu}(\gamma) Q_1, & x < 0, 
\end{cases}
\end{align}
and $G^h_\gamma$ is the piece associated to the hyperbolic projections, 
\begin{align}
G_\gamma^h (x) = \begin{cases}
- c_{2m}^{-1} P_1 e^{M(\gamma) x} P^\mathrm{ss} (\gamma) Q_1, & x \geq 0 \\
c_{2m}^{-1} P_1 e^{M(\gamma) x} P^\mathrm{uu} (\gamma) Q_1, & x < 0. 
\end{cases}
\end{align} 

This decomposition is natural in terms of $\gamma$ dependence, since it isolates the pieces of $G_\gamma^+$ which have a singularity at $\gamma = 0$. However, this decomposition is not natural from the point of view of spatial regularity: for $\gamma^2$ to the right of the essential spectrum, the total Green's function $G_\gamma^+$ belongs to $H^{2m-1} (\R)$, but for instance $G^\mathrm{heat}_\gamma$ is only in $H^1 (\R)$. In order to prove Proposition \ref{p: right resolvent estimate}, we will need estimates on derivatives of $G_\gamma^+$ up to order $2m-1$. Taking higher derivatives of the individual terms in the decomposition \eqref{e: right greens fcn decomposition} introduces terms involving the Dirac delta and its derivatives, since these terms have only one classical derivative at $x = 0$. However, because $G_\gamma^+ \in H^{2m-1} (\R)$, these distribution-valued terms arising from derivatives of $G^\mathrm{heat}_\gamma, G^c_\gamma - G^\mathrm{heat}_\gamma$, and $\tilde{G}^c_\gamma + G^h_\gamma$ up to order $2m-1$ must disappear when added together. Therefore, when estimating these derivatives, it suffices for our purposes to disregard the singular parts, as they give no contribution to the end result in Proposition \ref{p: right resolvent estimate}. 

In light of this, for any function $g \in H^{2m-1} (\R)$ which is smooth away from $x = 0$, for any integer $1 \leq k \leq 2m-1$, we define an operator $\tilde{\partial}_x^k$ returning only the regular part of the derivative, which is of course given by the piecewise derivative
\begin{align*}
\tilde{\partial}_x^k g(x) = \begin{cases}
\partial_x^k g(x), & x > 0, \\
\partial_x^k g(x), & x < 0.
\end{cases}
\end{align*}

In order to show that $G_\gamma^+$ behaves like the heat resolvent, we first estimate the difference $G^c_\gamma - G^\mathrm{heat}_\gamma$, showing that the difference is $\mathrm{O}(\gamma)$ and therefore can be absorbed into our error term.
\begin{lemma}\label{l: center to heat resolvent}
	Let $\delta > 0$ be small. There exists a constant $C > 0$ such that if $\gamma^2$ is to the right of the essential spectrum of $\mcl$ and $|\gamma| \leq \delta$, then for any integer $0 \leq k \leq 2m-1$,
	\begin{align}
	|\tilde{\partial}_x^k G^c_\gamma(x) - \tilde{\partial}_x^k G^\mathrm{heat}_\gamma (x)| \leq C |\gamma| \langle x \rangle.
	\end{align}
\end{lemma}
\begin{proof}
	Let $x \geq 0$, and first suppose $|\gamma^2 x| < 2$. We write
	\begin{align*}
	|G_\gamma^c(x) - G_\gamma^\mathrm{heat}(x)| = \left| \frac{c_{2m}^{-1} \beta}{\gamma} \right| | e^{\nu^- (\gamma) x} - e^{- \nu_0 \gamma x}| = \frac{C}{|\gamma|} |e^{-\nu_0 \gamma x}| | e^{ (\nu^-(\gamma) + \nu_0 \gamma) x} - 1|. 
	\end{align*}
	Since 
	\begin{align*}
	\nu^- (\gamma) = -\nu_0 \gamma + \mathrm{O}(\gamma^2)
	\end{align*}
	we know that 
	\begin{align*}
	|\nu^- (\gamma) x + \nu_0 \gamma x| \leq C |\gamma^2 x| \leq 2C
	\end{align*}
	for some constant $C > 0$. It follows from differentiability of the exponential function that 
	\begin{align*}
	|e^{ (\nu^-(\gamma) + \nu_0 \gamma) x} - 1| \leq C |(\nu^- (\gamma) + \nu_0 \gamma) x| \leq C |\gamma^2 x|. 
	\end{align*}
	Also,  $\Re \gamma \geq 0$ implies $e^{- \nu_0 \gamma x}$ is bounded. Hence we have 
	\begin{align*}
	|G^c_\gamma (x) - G^\mathrm{heat}_\gamma (x)| \leq C |\gamma| \langle x \rangle,
	\end{align*}
	for $x > 0$ and $|\gamma^2 x| < 2$. Next, we assume $|\gamma^2 x| \geq 2$. Then, since $|e^z| \leq 1$ for $\Re z \leq 0$, and $\Re \nu^- (\gamma) \leq 0$ for $\gamma^2$ to the right of the essential spectrum, we have 
	\begin{align*}
	|G^c_\gamma(x) - G^\mathrm{heat}_\gamma (x)| \leq \frac{C}{|\gamma|} |e^{\nu^-(\gamma) x} - e^{- \nu_0 \gamma x}| \leq 2 \frac{C}{|\gamma|} \leq \frac{C}{|\gamma|} |\gamma^2 x| \leq C |\gamma| \langle x \rangle. 
	\end{align*}
	Hence we have the desired estimate in all cases, for $x \geq 0$. The argument for $x < 0$ is completely analogous, as are the estimates on the regular parts of the derivatives. 
\end{proof}

To prove the second part of Proposition \ref{p: right resolvent estimate}, we will also need to control the difference between $G^c_\gamma - G^\mathrm{heat}_\gamma$ and the leading order term in $\gamma$ in this expression. Fixing $x$ and expanding formally, one finds 
\begin{align}
G^c_\gamma(x) - G^\mathrm{heat}_\gamma(x) = -c_{2m}^{-1} \beta \gamma h(x) + \mathrm{O}(\gamma^2),
\end{align}
where 
\begin{align*}
h(x) = \begin{cases}
\nu_2^- x, &x \geq 0, \\
\nu_2^+ x, &x < 0,
\end{cases}
\end{align*}
and where $\nu^\pm (\gamma) = \pm \frac{1}{\sqrt{\alpha}} \gamma + \nu_2^\pm \gamma^2 + \mathrm{O}(\gamma^3)$. We now show precisely that the $\mathrm{O}(\gamma^2)$ term in this expression is appropriately controlled in space, and so contributes to the error term in \eqref{e: right second order expansion}. 

\begin{lemma}\label{l: center to heat resolvent expansion}
	Let $\delta > 0$ be small. There exists a constant $C > 0$ such that if $\gamma^2$ is to the right of $\Sigma^+_{\eta_*}$ and $|\gamma| \leq \delta$, then for any integer $0 \leq k \leq 2m-1$,
	\begin{align}
	|\tilde{\partial}_x^k (G^c_\gamma (x) - G^\mathrm{heat}_\gamma (x) + c_{2m}^{-1} \beta \gamma h(x))| \leq C |\gamma|^2 \langle x \rangle^2. \label{e: gc minus gheat expansion}
	\end{align}
\end{lemma}
\begin{proof}
	We focus on proving \eqref{e: gc minus gheat expansion} for $k = 0$, since the estimates on the regular parts of higher derivatives are similar. We only show the case where $x > 0$, since $x < 0$ is similar. For $x > 0$, we have 
	\begin{align*}
	|G^c_\gamma (x) - G^\mathrm{heat}_\gamma (x) + c_{2m}^{-1} \beta \gamma h(x)| = C \left| \frac{1}{\gamma} e^{\nu^-(\gamma) x} - \frac{1}{\gamma} e^{-\nu_0 \gamma x} - \nu_2^- \gamma x \right|. 
	\end{align*}
	Since 
	\begin{align*}
	\left| \nu^-_2 \gamma x - \frac{(\nu^-(\gamma) +\nu_0 \gamma)}{\gamma} x \right| \leq C |\gamma|^2 |x|, 
	\end{align*}
	we may replace $\nu_2^- \gamma$ in this expression with $(\nu^-(\gamma)+\nu_0 \gamma)/\gamma$ and absorb the difference into the error term. We let $z = \gamma x$, and $w = (\nu^-(\gamma) + \nu_0 \gamma) x$. Note that for $\gamma$ small, $|w| \leq C |\gamma| |z| \leq C |z|$. Hence 
	\begin{align*}
	\frac{1}{|\gamma|^2 \langle x \rangle^2} \left| \frac{1}{\gamma} e^{\nu^-(\gamma) x} - \frac{1}{\gamma} e^{-\nu_0 \gamma x} - \frac{(\nu^-(\gamma) +\nu_0 \gamma)}{\gamma} x \right| &\leq \frac{|w|}{|\gamma|} \frac{1}{|z|^2} \left| e^{-\nu_0 z} \frac{(e^w-1)}{w} - 1 \right| \\
	&\leq \frac{C}{|z|} \left| e^{-\nu_0 z} (1+ \mathrm{O}(w)) - 1 \right| \\
	&\leq \frac{C}{|z|} \left( |e^{-\nu_0 z} - 1| + C |w| |e^{- \nu_0 z}| \right) \\
	&\leq C 
	\end{align*}
	for $z, w$ small. The expression is also bounded for $z, w$ large: the only term which appears potentially problematic is $|e^{-\nu_0 z} e^w| = |e^{\nu^-(\gamma) x}|$, which is bounded since $\gamma^2$ is to the right of the essential spectrum, so $\Re \nu^- (\gamma) \leq 0$. Hence we obtain \eqref{e: gc minus gheat expansion}. 
\end{proof}

We now estimate the remaining error terms in the decomposition of the Green's function. 
\begin{lemma}\label{l: greens fcn remainder estimates}
	Let $r > 3/2$. There is a constant $C > 0$ such that the remainder terms in the Green's function satisfy the estimate
	\begin{align}
	\|[\tilde{\partial}_x^k (\tilde{G}^c_\gamma + G^h_\gamma - \tilde{G}^c_0 - G^h_0)] \ast g\|_{L^2_{-r, -r}} \leq C |\gamma| \|g\|_{L^2_{r,r}}  \label{e: right greens fcn error terms}
	\end{align}
	for any integer $1 \leq k \leq 2m-1$, any $g \in L^2_r (\R)$, and any $\gamma$ sufficiently small with $\gamma^2$ to the right of $\Sigma^+_{\eta_*}$. 

	Furthermore, if $r > 5/2$, then we can expand to second order in the sense that there is a function $\tilde{G}^1$ such that 
	\begin{align}
	\|[\tilde{\partial}_x^k (\tilde{G}^c_\gamma + G^h_\gamma - \tilde{G}^c_0 - G^h_0 - \gamma \tilde{G}^1)] \ast g\|_{L^2_{-r, -r}} \leq C |\gamma|^2 \|g\|_{L^2_{r, r}} \label{e: right greens fcn error terms 2nd order}
	\end{align}
	for any integer $1 \leq k \leq 2m-1$, any $g \in L^2_{r,r} (\R)$, and any $\gamma$ sufficiently small with $\gamma^2$ to the right of $\Sigma^+_{\eta_*}$. 
\end{lemma}

\begin{proof}
	We focus on the estimate \eqref{e: right greens fcn error terms} for $k = 0$, since the estimates on the regular parts of the derivatives are analogous. Note that for $\gamma$ small, $G_\gamma^h$ is analytic in $\gamma$ and is exponentially localized in space, with decay rate independent of $\gamma$. It follows that $\gamma \mapsto G_\gamma^h$ is analytic from a neighborhood of the origin into $L^1 (\R)$. Young's convolution inequality then implies that convolution with $G^h_\gamma$ is analytic in $\gamma$ as a family of bounded operators on $L^2(\R)$, and so in particular
	\begin{equation*}
	\|(G^h_\gamma - G^h_0) \ast g\|_{L^2_{-r, -r}} \leq\|(G^h_\gamma - G^h_0) \ast g\|_{L^2} \leq C |\gamma| \|g\|_{L^2} \leq C |\gamma| \|g\|_{L^2_{r, r}}
	\end{equation*}
	For the other term, we use the fact that for $\gamma$ small with $\gamma^2$ to the right of the essential spectrum, we have $\Re \nu^- (\gamma) \leq 0$, and so for $x > 0$
	\begin{align*}
	|e^{\nu^-(\gamma) x} - 1| \leq C |\nu^-(\gamma)| |x| \leq  C |\gamma| |x|, 
	\end{align*}
	and similarly for $x < 0$
	\begin{align*}
	|e^{\nu^+(\gamma) x} - 1| \leq C |\nu^+(\gamma)| |x| \leq C |\gamma| |x|, 
	\end{align*}
	using the estimate $|e^z - 1| \leq C |z|$ for $\Re z \leq 0$. This estimate together with the fact that the maps $\gamma \mapsto \tilde{P}^\mathrm{cs/cu}(\gamma)$ are analytic in $\gamma$ in a neighborhood of the origin imply that 
	\begin{align*}
	|\tilde{G}^c_\gamma (x) - \tilde{G}^c_0 (x)| \leq C |\gamma| |x|. 
	\end{align*}
	The function space estimate in \eqref{e: right greens fcn error terms} then follows from the Cauchy-Schwarz inequality --- see the proof of Proposition \ref{p: right resolvent estimate} below. The proof of \eqref{e: right greens fcn error terms 2nd order} is similar, simply requiring Taylor expanding the exponential to higher order. 
\end{proof}

The behavior of the heat resolvent improves when acting on odd functions $g$, compared to a generic function with the same localization. Restricting to odd functions in the resolvent equation $(\partial_{xx} - \gamma^2) u = g$ is equivalent to posing the problem on a half-line with a homogeneous Dirichlet boundary condition. The improved properties of the resolvent in this context have been  exploited in \cite{JensenHalfline} to establish expansions for resolvents of Schr\"odinger operators on the half-line. As in \cite{JensenHalfline}, we write for a sufficiently localized odd function $g$,
\begin{align*}
G^\mathrm{heat}_\gamma \ast g (x) = - c_{2m}^{-1} \beta \int_0^\infty G_\gamma^\mathrm{odd}(x,y) g(y) \, dy,
\end{align*}
where 
\begin{align}
G^\mathrm{odd}_\gamma (x,y) = \frac{1}{\gamma} \left( e^{- \nu_0 \gamma |x-y|} - e^{- \nu_0 \gamma |x+y|} \right). \label{e: G odd}
\end{align}
We collect the properties of $G^\mathrm{odd}_\gamma$ in the following lemma, whose proof follows from careful but elementary computation, similar to the proof of Lemma \ref{l: center to heat resolvent expansion}. 

\begin{lemma}\label{l: dirichlet resolvent}
	There exists a constant $C > 0$ such that for all $\gamma$ with $\Re \gamma \geq 0$, we have 
	\begin{align*}
	|G_\gamma^\mathrm{odd} (x,y) - 2 \nu_0 \min (x,y)| &\leq C |\gamma| \langle x \rangle \langle y \rangle, \\
	|\partial_x G_\gamma^\mathrm{odd} (x,y) - 2 \nu_0 \partial_x \min (x,y)| &\leq C |\gamma| \langle x \rangle \langle y \rangle,
	\end{align*}
	and 
	\begin{align*}
	|G^\mathrm{odd}_\gamma (x,y) - 2 \nu_0 \min (x,y) + 2 \gamma \nu_0^2 x y | &\leq C |\gamma|^2 \langle x \rangle^2 \langle y \rangle^2, \\
	|\partial_x (G^\mathrm{odd}_\gamma (x,y) - 2 \nu_0 \min (x,y) + 2 \gamma \nu_0^2 x y)| & \leq C |\gamma|^2 \langle x \rangle^2 \langle y \rangle^2. 
	\end{align*}
\end{lemma}

\begin{proof}[Proof of Proposition 2.1]
	Since $G_\gamma^+ \in H^{2m-1}_\mathrm{loc} (\R)$ for $\gamma^2$ to the right of the essential spectrum, for any integer $1 \leq k \leq 2m-1$, we may write 
	\begin{align*}
	\partial_x^k \int_\R G_\gamma (x-y) g(y) \, dy = \int_\R \partial_x G_\gamma (x-y) g(y) \, dy = \int_\R \tilde{\partial}_x^k G_\gamma (x-y) g(y) \, dy.
	\end{align*}
	Now that we have used regularity of $G_\gamma^+$ to replace the derivatives with only the regularized parts, we split $G_\gamma^+$ into its components as in \eqref{e: right greens fcn decomposition}, 
	\begin{align*}
	\int_\R \tilde{\partial}_x^k G_\gamma (x-y) g(y) \, dy = [\tilde{\partial}_x^k( G^\mathrm{heat}_\gamma + G^c_\gamma - G^\mathrm{heat}_\gamma + \tilde{G}^c_\gamma + G^h_\gamma)] \ast g (x). 
	\end{align*}
	By Lemma \ref{l: center to heat resolvent} we have 
	\begin{align*}
	|[\tilde{\partial}_x^k (G^c_\gamma - G^\mathrm{heat}_\gamma)] \ast g (x)| \leq C |\gamma| \int_\R |x-y| |g(y)| \, dy 
	\leq C |\gamma| \int_\R \max(\langle x \rangle, \langle y \rangle) |g(y)| \, dy. 
	\end{align*}
	For $g \in L^2_r (\R)$, we use the Cauchy-Schwarz inequality to obtain 
	\begin{align*}
	\|[\tilde{\partial}_x^k (G^c_\gamma - G^\mathrm{heat}_\gamma)] \ast g\|_{L^2_{-r,-r}} \leq C |\gamma| \|g\|_{L^2_{r,r}} \left( \int_\R \max(\langle x \rangle, \langle y \rangle)^2 (\langle x \rangle \langle y \rangle)^{-2r} dx dy \right)^{1/2}. 
	\end{align*}
	Splitting this integral into integrals over regions $|y| \leq |x|$ and $|x| \leq |y|$, one finds that the integral is finite for $r > 3/2$, and one thereby obtains 
	\begin{align*}
	\|[\tilde{\partial}_x^k (G^c_\gamma - G^\mathrm{heat}_\gamma)] \ast g\|_{L^2_{-r, -r}} \leq C |\gamma| \|g\|_{L^2_{r,r}}. 
	\end{align*}
	Hence this term is $\mathrm{O}(\gamma)$, and can be absorbed into the error term. In proving \eqref{e: right second order expansion}, one instead uses the estimate in Lemma \ref{l: center to heat resolvent expansion}, which gives an expansion of this term to second order in $\gamma$. 
	
	Expansions for $\tilde{\partial}_x^k (\tilde{G}^c_\gamma + G^h_\gamma)$ are already given in Lemma \ref{l: greens fcn remainder estimates}, so it only remains to obtain expansions for $\tilde{\partial}_x^k G^\mathrm{heat}_\gamma$ acting on odd functions $g$. For $k = 0$ or 1 these expansions follows immediately from the estimates in Lemma \ref{l: dirichlet resolvent}. For $k \geq 2$, the estimates are actually simpler, and can be seen directly from $G^\mathrm{heat}_\gamma$ rather than using the odd extension, since taking derivatives in $x$ introduces extra factors of $\gamma$. This completes the proof of Proposition \ref{p: right resolvent estimate}. 
\end{proof}

We conclude this section by observing that our spectral assumptions imply that $(\mcl_- - \gamma^2)^{-1}$ is analytic in $\gamma^2$. 

\begin{lemma}\label{l: left resolvent estimates}
	For $\eta \geq 0$ sufficiently small, the operator $(\mcl_- - \gamma^2)^{-1}: L^2_{\exp, \eta} (\R) \to H^{2m-1}_{\mathrm{exp}, \eta} (\R)$ is analytic in $\gamma^2$ in a neighborhood of the origin. 
\end{lemma}
\begin{proof}
	By standard spectral theory, this amounts to saying that $0$ is in the resolvent set of the operator $\mathcal{L}_-$, which follows directly from Hypothesis \ref{hyp: left}, and the fact that the Fredholm borders in the exponentially weighted space depend continuously on the parameter $\eta$. 
\end{proof}

\section{Full resolvent estimates}\label{s: full resolvent}

\subsection{Far-field/core decomposition and leading order estimates} \label{ss: farfield core leading order}
We now extend the resolvent estimates of Proposition \ref{p: right resolvent estimate} to the full resolvent operator $(\mcl - \gamma^2)^{-1}$, in the following sense. Note that we only require additional algebraic localization on the right. 

\begin{prop}\label{p: full resolvent order 1}
	Let $r > 3/2$. There are constants $C > 0$ and $\delta > 0$ such that for any $g \in L^2_r (\R)$, the solution to $(\mcl - \gamma^2) u = g$ satisfies 
	\begin{align}
	\|u(\gamma) - u(0)\|_{H^{2m-1}_{-r}} \leq C |\gamma| \|g\|_{L^2_r}
	\end{align}
	for all $\gamma \in B(0,\delta)$ with $\gamma^2$ to the right of $\Sigma^+_{\eta_*}$. 
\end{prop}

If this proposition holds, we write $(\mcl - \gamma^2)^{-1} = R_0 + \mathrm{O}(\gamma)$ in $\mathcal{B}(L^2_r (\R), H^{2m-1}_{-r}(\R))$. The aim of our approach is to first solve on the left and on the right with the asymptotic operators by decomposing the data and the solution appropriately, leaving an equation on the center $(\mcl-\gamma^2) u^c = \tilde{g}$ with exponentially localized data. We then solve this equation with a far-field/core decomposition as in \cite{PoganScheel} to obtain our estimates.

Specifically, we let $(\chi_-, \chi_c, \chi_+)$ be a partition of unity on $\R$, with $\chi_+$ satisfying \eqref{e: chiplus} and $\chi_- (x) = \chi_+ (-x)$, so that $\chi_c$ is compactly supported. We use this partition of unity to decompose our data $g$ into a ``left piece'', a ``center piece'', and a ``right piece'' by writing 
\begin{align*}
g = \chi_- g + \chi_c g + \chi_+ g =: g_- + g_c + g_+. 
\end{align*}
We would like to decompose our solution accordingly into $u = u^- + u^c + u^+$, with $u^-$ and $u^+$ solving $(\mcl_\pm-\gamma^2) u^\pm = g_\pm$, and with the remaining piece $(\mcl - \gamma)^2 u^c = g_c$ having strongly localized data. However, we need to refine this decomposition slightly in order to obtain sharp estimates. As we saw in Section \ref{sec: asymptotic resolvents}, the behavior of $(\mcl_+ - \gamma^2)^{-1}$ is much improved when acting on odd functions. Therefore, we let $g^\mathrm{odd}_+(x) = g_+ (x) - g_+(-x)$ be the odd part of $g_+$, and let $u_+$ be the solution to 
\begin{align}
(\mcl_+ - \gamma^2) u^+ = g_+^\mathrm{odd}.  \label{e: ff core right}
\end{align}
We let $u^-$ be the solution to
\begin{align}
(\mcl_- - \gamma^2) u^- = g_-.  \label{e: ff core left}
\end{align}
We decompose the solution $u$ to $(\mcl - \gamma^2) u = g$ as $u = u^- + u^c + \chi_+ u^+$. The additional cutoff function on $u^+$ is so that we do not have to require algebraic localization on the left when using Proposition \ref{p: right resolvent estimate}. After a short computation, one finds that $u^c$ must solve
\begin{align}
(\mcl - \gamma^2)u^c = \tilde{g}(\gamma), \label{e: u c eqn}
\end{align}
where 
\begin{align}
\tilde{g}(\gamma) := g_c + (\chi_+ - \chi_+^2) g - [\mcl_+, \chi_+] u^+ + (\mcl_+ - \mcl) (\chi_+ u^+) + (\mcl_- - \mcl)u^-, \label{e: tilde f def}
\end{align}
and $[\mcl_+, \chi_+]$ is the commutator 
\begin{align*}
[\mcl_+, \chi_+] u^+ = \mcl_+ (\chi_+ u^+) - \chi_+ (\mcl_+ u^+). 
\end{align*}
Note that $\tilde{g}(\gamma)$ is exponentially localized on the right, so that we may solve this equation using a far-field/core decomposition, taking advantage of the fact that $\mcl$ is a Fredholm operator on exponentially weighted spaces with small weights. The right hand side $\tilde{g}$ depends on $\gamma$ through $u^+$ and $u^-$, and we use the estimates in Section \ref{sec: asymptotic resolvents} to characterize this dependence in the following lemma. 
\begin{lemma}\label{l: farfield core rhs control}
	Let $r > 3/2$, and let $\eta > 0$ be small. For $\gamma$ small with $\gamma^2$ to the right of $\Sigma^+_{\eta_*}$, we have $\tilde{g} (\gamma) \in L^2_{\exp, \eta} (\R)$, and 
	\begin{align}
	\|\tilde{g}(\gamma) - \tilde{g}(0)\|_{L^2_{\exp, \eta}} \leq C |\gamma| \| g\|_{L^2_r}. 
	\end{align}
\end{lemma}
\begin{proof}
	The terms $g_c$ and $(\chi_+ - \chi_+^2) g$ in \eqref{e: tilde f def} are independent of $\gamma$ and are compactly supported by construction. The commutator $[\mcl_+, \chi_+]$ is a differential operator of order $2m-1$ with smooth compactly supported coefficients, since $\chi_+$ is constant outside a compact set, so $[\mcl_+, \chi_+] u^+$ is also compactly supported. Similarly, $(\mcl_+ - \mcl) (\chi_+ \cdot)$ is a differential operator of order $2m-1$ whose coefficients converge to zero exponentially quickly as $x \to \infty$, and are identically zero for $x$ negative. Hence, if $\eta$ is sufficiently small, 
	\begin{align*}
	\|\omega_\eta \left( - [\mcl_+, \chi_+] + (\mcl_+ - \mcl) \chi_+ \right) (u^+(\gamma) - u^+(0))\|_{L^2} &\leq C \| (u^+(\gamma) - u^+(0))\|_{H^{2m-1}_{-r, -r}} \\
	&\leq C |\gamma| \|g_+^\mathrm{odd}\|_{L^2_{r, r}} \\
	&\leq C |\gamma| \|g\|_{L^2_r}, 
	\end{align*}
	by Proposition \ref{p: right resolvent estimate}. Similarly, 
	\begin{align*}
	\|\omega_\eta (\mcl_- - \mcl) (u^-(\gamma) - u^-(0)) \|_{L^2} \leq C |\gamma| \|g_-\|_{L^2_{\mathrm{exp}, \eta}} \leq C |\gamma| \|g\|_{L^2_r}, 
	\end{align*}
	by Lemma \ref{l: left resolvent estimates}, using the fact that $g_-$ is supported only on the left, so the exponential weight on the right can be replaced by an algebraic weight.
\end{proof}

We now solve $(\mcl - \gamma^2) u^c = \tilde{g}$ by making the far-field/core ansatz
\begin{align}
u^c (x) = w(x) + a \chi_+ (x) e^{\nu^- (\gamma) x}, 
\end{align}
where $w \in H^{2m}_{\mathrm{exp, \eta}} (\R)$ is exponentially localized, $a \in \C$ is a complex parameter, and $\nu^- (\gamma)$ is the spatial eigenvalue given in \eqref{e: spatial evals}.  With this ansatz, the equation $(\mcl - \gamma^2) u^c = \tilde{g}$ becomes 
\begin{align}
F(w, a; \gamma) := \mcl w + a \mcl \left( \chi_+ e^{\nu^-(\gamma) \cdot} \right) - \gamma^2 (w + a \chi_+ e^{\nu^- (\gamma) \cdot}) = \tilde{g} \label{e: farfield core F def},
\end{align}
with the goal of solving for $w$ and $a$ with $\tilde{g}$ and $\gamma$ as variables. By Hypothesis \ref{hyp: spreading speed} and Palmer's theorem, $\mcl: H^{2m}_{\mathrm{exp}, \eta} (\R) \subseteq L^2_{\mathrm{exp}, \eta} (\R) \to L^2_{\mathrm{exp, \eta}} (\R)$ is a Fredholm operator with index -1. The addition of the extra parameter $a$ makes $(w, a) \mapsto F(w, a; \gamma)$ a Fredholm operator with index 0 for $\gamma$ small, by the Fredholm bordering lemma \cite[Lemma 4.4]{ArndBjornRelativeMorse}. The parameter $a$ is introduced in a manner which precisely captures the far-field behavior of $\mcl$ at $x = \infty$, which ultimately allows us to recover invertibility of $\mcl$ in this sense in a neighborhood of $\gamma = 0$. 

\begin{lemma}\label{l: farfield core invertibility}
	There exists $\delta > 0$ such that the map $F: H^{2m}_{\mathrm{exp}, \eta} (\R) \times \C \times B(0,\delta) \to L^2_{\mathrm{exp}, \eta} (\R)$ is well-defined and $(w, a) \mapsto F(w,a; \gamma)$ is invertible. We denote the solutions $(w,a)$ to \eqref{e: farfield core F def} by $w(\cdot; \gamma) = T(\gamma) \tilde{g}$ and $a(\gamma) = A(\gamma) \tilde{g}$. The maps 
	\begin{align*}
	\gamma \mapsto T(\gamma) : B(0, \delta) \to \mathcal{B} \left( L^2_{\mathrm{exp}, \eta} (\R), H^{2m}_{\mathrm{exp}, \eta} (\R) \right) 
	\end{align*}
	and 
	\begin{align*}
	\gamma \mapsto A(\gamma) : B(0, \delta) \to \mathcal{B} \left( L^2_{\mathrm{exp}, \eta} (\R), \C \right) 
	\end{align*}
	are analytic in $\gamma$. 
\end{lemma}
\begin{proof}
	The fact that $F$ is well-defined and maps into $L^2_{\mathrm{exp}, \eta} (\R)$ follows from writing 
	\begin{align*}
	(\mcl - \gamma^2) (\chi_+ e^{\nu- (\gamma) \cdot} ) = \chi_+ (\mcl - \mcl_+) e^{\nu^- (\gamma) \cdot} + [\mcl, \chi_+] e^{\nu^- (\gamma) \cdot},
	\end{align*}
	using $(\mcl_+-\gamma^2) e^{\nu^-(\gamma) x} = 0$. The commutator $[\mcl, \chi_+]$ has compactly supported coefficients, and the coefficients of $\mcl-\mcl_+$ decay exponentially as $x \to \infty$, so both of these terms are exponentially localized uniformly in $\gamma$, and so $F$ maps into $L^2_{\mathrm{exp}, \eta} (\R)$. 
	
	Note next that $\gamma \mapsto F(\cdot, \cdot; \gamma)$ is analytic in $\gamma$ as a family of bounded operators. This is formally clear from the fact that $\nu^- (\gamma)$ is analytic in $\gamma$; for a rigorous justification, see the proof of Proposition 5.11 in \cite{PoganScheel}. Since we have already observed that $(w, a) \mapsto F(w,a; \gamma)$ is Fredholm with index 0 for $\gamma \in B(0,\delta)$ for some $\delta$ small, to prove the lemma it suffices by the analytic Fredholm theorem to check that $(w, a) \mapsto F(w,a; 0)$ is invertible. Since $(w,a) \mapsto F(w,a;0)$ is Fredholm index 0, we only need to check that $F(w,a; 0)$ has no kernel. Suppose that there is a kernel. Then, from \eqref{e: farfield core F def}, we have $\mcl (w + a \chi_+) = 0$
	for some $w \in H^{2m}_{\mathrm{exp}, \eta} (\R)$, $a \in \C$. The function $w + a \chi_+$ is bounded, so this implies $\mcl$ has a resonance at $0$, contradicting Hypothesis \ref{hyp: resonances}. Hence $(w,a) \mapsto F(w,a; 0)$ is invertible, and the lemma follows from the analytic Fredholm theorem. 
\end{proof}

\begin{proof}[Proof of Proposition \ref{p: full resolvent order 1}]
	By the above, the solution to $(\mcl - \gamma^2) u =g$ can be decomposed as $u = u^- + u^c + \chi_+ u^+$, where $u^-$, $u^+$, and $u^c$ solve \eqref{e: ff core left}, \eqref{e: ff core right} and \eqref{e: u c eqn} respectively. Lemma \ref{l: left resolvent estimates} and Proposition \ref{p: right resolvent estimate} imply the desired estimates for $u^-$ and $u^+$, so we only need to estimate the $\gamma$ dependence of $u^c$. By Lemma \ref{l: farfield core invertibility}, for $\gamma \in B(0,\delta)$, $u^c$ is given by 
	\begin{align*}
	u^c (\gamma) = T(\gamma) \tilde{g}(\gamma) + A(\gamma) \tilde{g}(\gamma) \chi_+ e^{\nu^-(\gamma) \cdot},
	\end{align*}
	and so 
	\begin{multline}
	\|u^c(\gamma) - u^c (0)\|_{H^{2m-1}_{-r}} \leq \|T(\gamma) \tilde{g}(\gamma) - T(0) \tilde{g}(0)\|_{H^{2m-1}_{-r}} \\ +  \|A(\gamma) \tilde{g}(\gamma) \chi_+ e^{\nu^- (\gamma) \cdot} - A(0) \tilde{g}(0) \chi_+\|_{H^{2m-1}_{-r}} \label{e: u center estimate}
	\end{multline}
	For the first term, we write
	\begin{align*}
	T(\gamma) \tilde{g}(\gamma) - T(0) \tilde{g}(0) = (T(\gamma)-T(0))\tilde{g}(\gamma) + T(0)(\tilde{g}(\gamma) - \tilde{g}(0)),
	\end{align*}
	and then estimate, using Lemma \ref{l: farfield core invertibility} to expand $T(\gamma)$ and Lemma \ref{l: farfield core rhs control} to control $\tilde{g}(\gamma)$,
	\begin{align*}
	\|(T(\gamma) - T(0)) \tilde{g}(\gamma)\|_{H^{2m-1}_{-r}} \leq C \|(T(\gamma) - T(0)) \tilde{g}(\gamma)\|_{H^{2m}_{\mathrm{exp}, \eta}} &\leq C |\gamma| \| \tilde{g}(\gamma) \|_{L^2_{\mathrm{exp}, \eta}} 
	\leq C|\gamma| \| g \|_{L^2_r}. 
	\end{align*}
	Similarly, we obtain 
	\begin{align*}
	 \|T(0)(\tilde{g}(\gamma)-\tilde{g}(0))\|_{H^{2m-1}_{-r}} \leq C \|T(0) (\tilde{g}(\gamma) - \tilde{g}(0))\|_{H^{2m}_{\mathrm{exp},\eta}} &\leq C \|\tilde{g}(\gamma) - \tilde{g}(0)\|_{L^2_{\mathrm{exp}, \eta}} 
	 \leq C |\gamma| \| g \|_{L^2_r},
	\end{align*}
	and so $\|T(\gamma) \tilde{g}(\gamma) - T(0) \tilde{g}(0)\|_{H^{2m-1}_{-r}} \leq C | \gamma | \| g \|_{L^2_r}$. 
	For the second term in \eqref{e: u center estimate}, we have 
	\begin{multline*}
	\hspace{-.2cm} \|A(\gamma) \tilde{g}(\gamma) \chi_+ e^{\nu^- (\gamma) \cdot} - A(0) \tilde{g} (0) \chi_+\|_{H^{2m-1}_{-r}} \leq \|e^{\nu^-(\gamma) \cdot} \chi_+ (A(\gamma) \tilde{g}(\gamma) - A(0) \tilde{g} (0))\|_{H^{2m-1}_{-r}} \\ + \|A(0) \tilde{g} (0) \chi_+ (1 - e^{\nu^- (\gamma) \cdot})\|_{H^{2m-1}_{-r}}.
	\end{multline*}
	Using Lemmas \ref{l: farfield core rhs control} and \ref{l: farfield core invertibility}, we obtain an estimate 
	\begin{align}
	|A(\gamma) \tilde{g}(\gamma) - A(0) \tilde{g}(0)| \leq C |\gamma| \|g\|_{L^2_r}. \label{e: A gamma minus A 0}
	\end{align}
	Since $e^{\nu^-(\gamma) x}$ is a bounded function for $\gamma^2$ to the right of the essential spectrum, and constants are controlled in $L^2_{-r}$ for $r > 1/2$, by \eqref{e: A gamma minus A 0} we conclude that
	\begin{align*}
	\|e^{\nu^-(\gamma) \cdot} \chi_+ (A(\gamma) \tilde{g}(\gamma) - A(0) \tilde{g} (0))\|_{H^{2m-1}_{-r}} \leq C |\gamma| \|g\|_{L^2_r}. 
	\end{align*}
	For the second term, we use the fact that $|1 - e^{\nu^-(\gamma) x}| \leq C |\nu^-(\gamma)| |x| \leq C |\gamma| |x|$ for $\gamma^2$ to the right of the essential spectrum. This term is controlled in $L^2_{-r}$ for $r>3/2$, so we have 
	\begin{align*}
	\|A(0)\tilde{g}(0) \chi_+ (1-e^{\nu^-(\gamma) \cdot})\|_{L^2_{-r}} \leq C |\gamma| \|g\|_{L^2_r}.
	\end{align*}
	The estimates on the derivatives in this term are easier, since taking derivatives gains factors of $\gamma$, and we can control $e^{\nu^- (\gamma) x}$ in $L^2_{-r}$ for $r > 1/2$. This completes the proof of the proposition. 
\end{proof}

\subsection{Higher order expansions and asymptotics of the Green's function}
The regularity of the resolvent obtained in Proposition \ref{p: full resolvent order 1} is sufficient to prove Theorem \ref{t: stability}, but in order to obtain the asymptotic description of the solution in Theorem \ref{t: asymptotics}, we need to expand the resolvent to higher order, in spaces of higher algebraic localization. Integrating along the contour that we will choose in Section \ref{s: semigroup estimates} will reveal that the part of the semigroup associated to the term $R_0$ in the expansion $(\mcl-\gamma^2)^{-1} = R_0 + \gamma R_1 + \mathrm{O}(\gamma^2)$ decays exponentially in time, and so the $t^{-3/2}$ decay stems from the term $\gamma R_1$. Hence, to identify the asymptotics of the solution, we both need to expand to higher order and identify the operator $R_1$. The first task proceeds as in Section \ref{ss: farfield core leading order}, simply keeping track of higher order $\gamma$ dependence using the relevant results from Section \ref{sec: asymptotic resolvents}, so we state these results without proof. To characterize $R_1$, we adapt our far-field/core approach to solve $(\mcl-\gamma^2) G_\gamma = - \delta_y$, constructing the resolvent kernel $G_\gamma$, and expanding it in $\gamma$ to determine $R_1$. 

\begin{lemma}\label{l: farfield core rhs 2nd order}
	Let $r > 5/2$, and let $\eta > 0$ be small. For $\gamma$ small with $\gamma^2$ to the right of $\Sigma^+_{\eta_*}$, we have $\tilde{g}(\gamma) \in L^2_{\mathrm{exp}, \eta} (\R)$, and 
	\begin{align*}
	\|\tilde{g} (\gamma) - \gamma \tilde{g}_1 - \tilde{g}(0)\|_{L^2_{\mathrm{exp}, \eta}} \leq C |\gamma|^2 \|g\|_{L^2_r}
	\end{align*}
	for some $\tilde{g}_1 \in L^2_{\exp, \eta} (\R)$. 
\end{lemma}

Using Lemma \ref{l: farfield core rhs 2nd order}, we obtain the following refinement of Proposition \ref{p: full resolvent order 1}

\begin{prop}\label{p: full resolvent order 2}
	Let $r > 5/2$. There are constants $C > 0$ and $\delta > 0$ and an operator $R_1 : L^2_r (\R) \to H^{2m-1}_{-r} (\R)$ such that for any $g \in L^2_r (\R)$, the solution to $(\mcl - \gamma^2) u = g$ satisfies 
	\begin{align}
	\|u(\gamma) - \gamma u^1 - u(0)\|_{H^{2m-1}_{-r}} \leq C |\gamma| \|g\|_{L^2_r},
	\end{align}
	where $u^1 = R_1 g$, for all $\gamma \in B(0,\delta)$ with $\gamma^2$ to the right $\Sigma^+_{\eta_*}$. 
\end{prop}

To construct the resolvent kernel $G_\gamma$ with our far-field/core decomposition, we must view $F$ defined by \eqref{e: farfield core F def} as a map $F: H^{2m-1}_{\mathrm{exp}, \eta} (\R) \times \C \times B(0, \delta) \to H^{-1}_{\mathrm{exp}, \eta} (\R)$. First we show that $\mcl$ retains Fredholm properties when acting on these spaces. 

\begin{lemma}
	We can extend $\mcl$ to an operator from $H^{2m-1}_{\mathrm{exp}, \eta} (\R)$ to $H^{-1}_{\exp, \eta} (\R)$, and this operator is Fredholm with index -1. 
\end{lemma}
\begin{proof}
	First define $\tilde{\mcl} : H^{2m}_{\exp, \eta} (\R) \to L^2_{\exp, \eta} (\R)$ by
	\begin{align*}
	\tilde{\mcl} = \mcl + (\partial_x +1)^{-1} [\mcl, \partial_x+1]. 
	\end{align*}
	Using the fact that all derivatives of the coefficients of $\mcl$ are exponentially localized, one finds that $(\partial_x + 1)^{-1} [\mcl, \partial_x+1]$ is a compact operator from $H^{2m}_{\mathrm{exp}, \eta} (\R)$ to $L^2_{\mathrm{exp}, \eta} (\R)$, and so $\tilde{\mcl}$ is Fredholm with index $-1$ as a compact perturbation of $\mcl$. We then define $\bar{\mcl} : H^{2m-1}_{\mathrm{exp}, \eta} (\R) \to H^{-1}_{\mathrm{exp}, \eta} (\R)$ by 
	\begin{align*}
	\bar{\mcl} = (\partial_x+1)\tilde{\mcl} (\partial_x+1)^{-1}.
	\end{align*}
	One may readily verify that if $u \in H^{2m}_{\mathrm{exp}, \eta} (\R)$, then $\bar{\mcl} u = \mcl u$, and hence $\bar{\mcl}$ is an extension of $\mcl$. Since the operator $\partial_x+1: H^k_{\mathrm{exp}, \eta}(\R) \to H^{k-1}_{\mathrm{exp}, \eta} (\R)$ is invertible, $\bar{\mcl}$ is Fredholm with index -1, and so we have produced the desired extension. We now write $\mcl = \bar{\mcl}$, understanding that we are using this extension of $\mcl$. 
\end{proof}

Repeating the argument of Lemma \ref{l: farfield core invertibility} in these spaces, we find a solution to $(\mcl - \gamma^2)G_\gamma = - \delta_y$ with the form 
\begin{align}
G_\gamma (x, y) = w(x, y; \gamma) + a(y, \gamma) \chi_+ (x) e^{\nu^- (\gamma) x},. \label{e: G farfield core}
\end{align}
where $w(\cdot; y, \gamma) \in H^{2m-1}_{\mathrm{exp}, \eta} (\R)$ for some $\eta > 0$ small, and both $w$ and $a$ are analytic in $\gamma$. We therefore write $G_\gamma = G^0 + \gamma G^1 + \mathrm{O}(\gamma^2)$, for fixed $x$ and $y$. Since $G$ depends analytically on $\gamma$, $G^1$ must solve the equation $(\mcl - \gamma^2) G_\gamma = - \delta_y$ at order $\gamma$, which is 
\begin{align}
\mcl G^1(\cdot; y) = 0. \label{e: G1 eqn}
\end{align}
Expanding the right hand side of \eqref{e: G farfield core} in $\gamma$, one finds that $G^1$ is linearly growing at $\infty$, and localized on the left. As noted in Section \ref{ss: remarks}, there is only one solution, up to a constant multiple, to $\mcl u = 0$ which is linearly growing at $\infty$ and localized on the left. We denote this solution by $\psi$, fixing the normalization by requiring 
\begin{align}
\lim_{x\to \infty} \frac{\psi(x)}{x} = 1. \label{e: psi normalization}
\end{align}

Since $G^1$ solves \eqref{e: G1 eqn}, we conclude that $G^1$ must be proportional to $\psi$, but with constant allowed to depend on the parameter $y$, so we have 
\begin{equation}
G^1(x; y) = \psi (x) g^1(y) 
\end{equation}
for some function $g^1 (y)$. Altogether, since the expansion obtained in Proposition \ref{p: full resolvent order 2} and the solution given by integration against the resolvent kernel must agree for $\gamma^2$ to the right of $\Sigma^+_{\eta_*}$, we obtain the following lemma. 

\begin{lemma}\label{l: R1}
	The operator $R_1$ in the expansion 
	\begin{align*}
	(\mcl-\gamma^2)^{-1} = R_0 + \gamma R_1 + \mathrm{O}(\gamma^2) 
	\end{align*}
	in $\mathcal{B} (L^2_r (\R), H^{2m-1}_{-r} (\R))$ for $r > 5/2$ guaranteed by Proposition \ref{p: full resolvent order 2} is given by 
	\begin{equation*}
	R_1 g (x) = \psi (x) \int_\R g^1 (y) g(y) \, dy. 
	\end{equation*}
\end{lemma}

If \eqref{e: spatial eval gap} holds, then as noted in Section \ref{ss: remarks} we must have $\psi (x) = \omega_{\eta_*} (x) q_*'(x)$. We can achieve the normalization condition \eqref{e: psi normalization} for instance by translating $q_*$ appropriately, without loss of generality.  

\section{Linear semigroup estimates}\label{s: semigroup estimates}

We now use the regularity of the resolvent obtained in Section \ref{s: full resolvent} in order to prove that the linear semigroup $e^{\mcl t}$ has the desired $t^{-3/2}$ decay, the essential step in proving Theorem \ref{t: stability}. Since $\mcl$ is sectorial \cite{Lunardi}, it generates an analytic semigroup through the contour integral 
\begin{align}
e^{\mcl t} = -\frac{1}{2\pi i} \int_\Gamma e^{\gamma^2 t} (\mcl - \gamma^2)^{-1} \, d(\gamma^2)
\end{align}
for a suitably chosen contour $\Gamma$. By Hypothesis \ref{hyp: resonances}, $\mcl$ has no unstable point spectrum, so the essential spectrum is the only obstacle to shifting the integration contour. Hypothesis \ref{hyp: spreading speed} guarantees that in $\gamma$, the Fredholm border which touches the origin may be parametrized as 
\begin{align}
\gamma(a) = i \gamma_1 a + \gamma_2 a^2 + \mathrm{O}(a^3) \label{e: essential spectrum gamma expansion}
\end{align}
for some real constants $\gamma_1, \gamma_2$.  To obtain optimal decay rates, we use the regularity of the resolvent near the origin to integrate along a contour which is tangent to the essential spectrum, which reveals the $t^{-3/2}$ decay rate. 

\begin{prop}\label{p: linear time decay}
	Let $r > 3/2$. There is a constant $C > 0$ such that the semigroup $e^{\mcl t}$ satisfies for $t > 0$
	\begin{align}
	\|e^{\mcl t}\|_{L^2_r \to H^{2m-1}_{-r}} \leq \frac{C}{t^{3/2}}. 
	\end{align}
\end{prop}
\begin{proof}
	For $\eps > 0$, we define our integration contour near the origin by  
	\begin{align*}
	\Gamma^0_\eps = \{ \gamma(a) = i a +  c_2 a^2 + \eps : a \in [-a_*, a_*] \},
	\end{align*}
	where $a_* > 0$ is small, and $c_2$ is chosen so that the limiting contour 
	\begin{align}
	\Gamma^0_0 = \{ \gamma (a) = i a + c_2 a^2 : a \in [-a_*, a_*] \}
	\end{align}
	is tangent to the essential spectrum in the $\gamma$-plane, touching it only at $\gamma = 0$ and staying to the right of it otherwise. The existence of such a $c_2$ is guaranteed by \eqref{e: essential spectrum gamma expansion}. We define these contours in the $\gamma$ plane, since it is natural to integrate in $\gamma = \sqrt{\lambda}$ in order to use the regularity of the resolvent in $\gamma$. We then let $\Gamma^\pm_\eps$ be continuations of $\Gamma^0_\eps$ out to infinity along straight lines in the left half $\lambda$-plane: see Figure \ref{fig: contours} for a depiction of these contours. We let $\Gamma_\eps$ denote the positively oriented concatenation of $\Gamma^-_\eps, \Gamma^0_\eps$, and $\Gamma^+_\eps$. By Proposition $(\mcl-\gamma^2)^{-1}$ is continuous at $\gamma = 0$ in $\mathcal{B}(L^2_r (\R), H^{2m-1}_{-r} (\R))$. Since it is also continuous on its resolvent set, and the limiting contour $\Gamma_0$ touches the spectrum of $\mcl$ only at $\gamma = 0$, this guarantees that $(\mcl - \gamma^2)^{-1}$ is continuous up to $\Gamma_0$. Together with sectoriality of $\mcl$ to control the behavior at large $\lambda$, this guarantees that the limit 
	\begin{align*}
	\lim_{\eps \to 0^+} -\frac{1}{2 \pi i} \int_{\Gamma_\eps} e^{\gamma^2 t} (\mcl - \gamma^2)^{-1} 2 \gamma \, d \gamma 
	\end{align*}
	exists in $\mathcal{B}(L^2_r (\R), H^{2m-1}_{-r} (\R))$. Since for every $\eps > 0$ the contour $\Gamma_\eps$ is in the resolvent set of $\mcl$, the value of this integral is independent of $\eps > 0$ by Cauchy's integral theorem. Hence we may write the semigroup using the integral over the limiting contour 
	\begin{align}
	e^{\mcl t} &= -\frac{1}{\pi i} \int_{\Gamma_0} e^{\gamma^2 t} (\mcl - \gamma^2)^{-1}  \gamma \, d \gamma \nonumber \\
	&= -\frac{1}{\pi i} \int_{\Gamma^0_0} e^{\gamma^2 t} (\mcl - \gamma^2)^{-1} \gamma \, d \gamma -  \sum_{\iota = \pm} \frac{1}{\pi i} \int_{\Gamma^\pm_0} e^{\gamma^2 t} (\mcl - \gamma^2)^{-1} \gamma \, d \gamma. \label{e: semigroup formula}
	\end{align}
	\begin{figure}
		\centering
		\includegraphics[width = 1\textwidth]{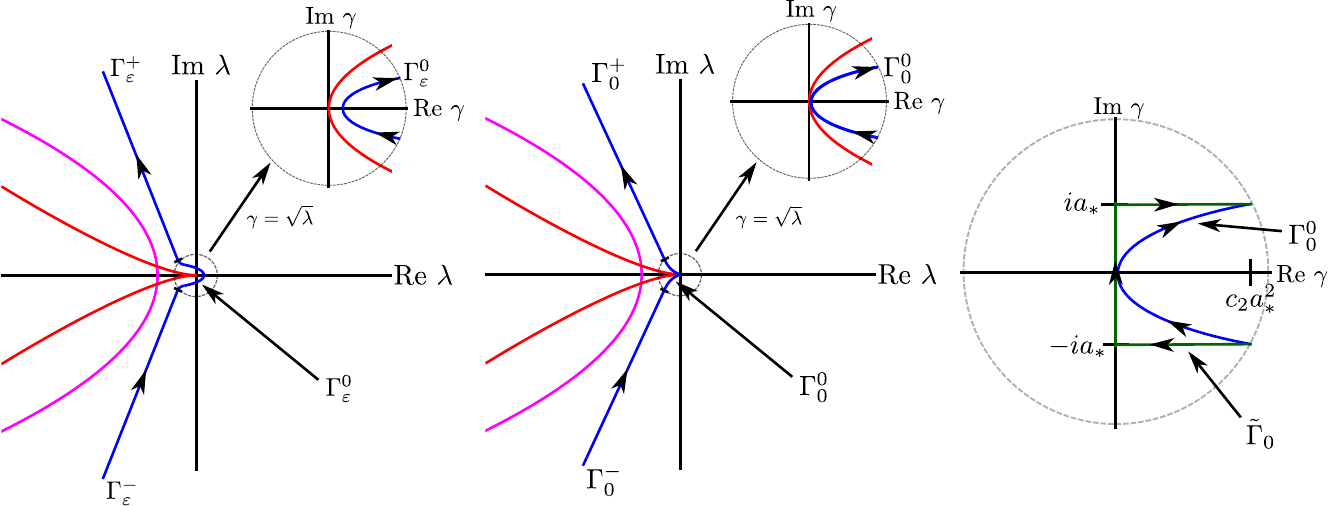}
		\caption{The Fredholm borders of $\mcl$ (magenta, red) together with our integration contours (blue), for $\eps > 0$ (left) and at the limit $\eps = 0$ (middle). The insets show the image of a neighborhood of the origin under the map $\gamma = \sqrt{\lambda}$. The rightmost inset shows the deformation of $\Gamma^0_0$ to $\tilde{\Gamma}_0$, the contour used in the proof of Proposition \ref{p: linear time asymptotics}.}
		\label{fig: contours}
	\end{figure}
	
	The integrals over $\Gamma^\pm_0$ are exponentially decaying in time, since each $\gamma^2$ along these contours is contained strictly in the left half plane and bounded away from the spectrum of $\mcl$. Using parabolic regularity \cite[Theorem 3.2.2]{Lunardi} to control the behavior of $(\mcl-\gamma^2)^{-1}$ for large $\gamma$, we readily obtain 
	\begin{align*}
	\left| \left| \frac{1}{\pi i} \int_{\Gamma_0^\pm} e^{\gamma^2 t} (\mcl - \gamma^2)^{-1} \gamma \, d \gamma \right| \right|_{L^2 \to H^{2m-1}} \leq C e^{-\mu t} 
	\end{align*}
	for some constants $C, \mu > 0$, which of course implies the same estimate in $L^2_r \to H^{2m-1}_{-r}$. 
	
	We now focus on the integral over $\Gamma^0_0$. We use Proposition \ref{p: full resolvent order 1} to write $(\mcl - \gamma^2)^{-1} = R_0 + \mathrm{O}(\gamma)$ in $\mathcal{B}(L^2_r (\R), H^{2m-1}_{-r} (\R))$ and explicitly parameterize the contour by $\gamma(a) = ia + c_2 a^2$ for $|a| \leq a_*$ to obtain
	\begin{align*}
	\frac{1}{\pi i} \int_{\Gamma_0^0} e^{\gamma^2 t} (\mcl - \gamma^2)^{-1} \gamma \, d \gamma &= \frac{1}{\pi i} \int_{-a_*}^{a_*} e^{\gamma(a)^2 t} (R_0 + \mathrm{O}(a)) \gamma(a) \gamma'(a) \, d a \\
	&= \frac{1}{\pi i} \int_{-a_*}^{a_*} \left[ \left( \frac{1}{2t} \partial_a e^{\gamma(a)^2 t} \right) R_0 + e^{\gamma (a)^2 t} E_0(a) \gamma (a) \gamma'(a) \right] da, 
	\end{align*}
	where we denote the $\mathrm{O}(a)$ terms by $E_0 (a)$. The first term is the integral of a total derivative, so
	\begin{align*}
	\frac{1}{\pi i} \int_{-a_*}^{a_*} \frac{1}{2t} \left( \partial_a e^{\gamma(a)^2 t} \right) R_0 \, da &= \frac{1}{2 \pi i} \frac{1}{t} R_0 \left( e^{\gamma(a_*)^2 t} - e^{\gamma(-a_*) t} \right) \\
	&= \frac{1}{2 \pi i} \frac{1}{t} R_0 e^{(-a_*^2 + c_2^2 a_*^4) t} \left( e^{2 i c_2 a_*^3 t} - e^{- 2 i c_2 a_*^3 t} \right) 
	\end{align*}
	We choose $a_*$ small enough so that $c_2^2 a_*^4 < \frac{a_*^2}{2}$ and hence 
	\begin{align}
	\left| \left| \frac{1}{2 \pi i} \frac{1}{t} \int_{-a_*}^{a_*}  \left( \partial_a e^{\gamma(a)^2 t} \right) R_0 \, da \right| \right|_{L^2_r \to H^{2m-1}_{-r}} \leq \frac{C}{t} e^{-a_*^2 t/2} \|R_0\|_{L^2_r \to H^{2m-1}_{-r}} \leq \frac{C}{t^{3/2}} \label{e: semigroup R0 estimate}
	\end{align}
	for $t > 0$. In fact, this contribution is exponentially decaying for $t$ large. We now estimate the second integral 
	\begin{align*}
	\left| \left|\int_{-a_*}^{a_*} e^{\gamma (a)^2 t} E_0(a) \gamma (a) \gamma'(a) \, da \right| \right|_{L^2_r \to H^{2m-1}_{-r}} &= \left| \left| \int_{-a_*}^{a_*} e^{(-a^2 + c_2^2 a^4) t} e^{2 i c_1 a^3 t} E_0 (a) \gamma (a) \gamma'(a) \, da \right| \right|_{L^2_r \to H^{2m-1}_{-r}}\\
	&\leq C \int_{-a_*}^{a_*} e^{-\frac{a^2}{2} t} |a|^2 \, da,
	\end{align*}
	for $a_*$ small. Changing variables to $z = \frac{a}{\sqrt{2}} \sqrt{t}$, we obtain 
	\begin{align*}
	\int_{-a_*}^{a_*} e^{- \frac{a^2}{2} t} a^2 \, da = \frac{C}{t^{3/2}} \int_{-a_* \sqrt{t/2}}^{a_* \sqrt{t/2}} e^{-z^2} z^2 \, dz \leq \frac{C}{t^{3/2}},
	\end{align*}
	which completes the proof of the proposition. 
\end{proof}

We now use the higher regularity of the resolvent obtained in Proposition \ref{p: full resolvent order 2} to identify the leading order asymptotics of $e^{\mcl t}$ as $t \to \infty$ by focusing on the term $\gamma R_1$ in the contour integral, since we have shown that the term associated to $R_0$ decays exponentially. 

\begin{prop}\label{p: linear time asymptotics}
	Let $r > 5/2$. Then the semigroup $e^{\mcl t}$ has the asymptotic expansion
	\begin{align}
	e^{\mcl t} = \frac{1}{2 \sqrt{\pi}} \frac{R_1}{t^{3/2}} + \mathrm{O}(t^{-2}) 
	\end{align}
	as $t \to \infty$, in $\mathcal{B}(L^2_r (\R), H^{2m-1}_{-r} (\R))$. 
\end{prop}
\begin{proof}
	We proceed as in the proof of Proposition \ref{p: linear time decay}, using the same integration contour $\Gamma_0$. Using Proposition \ref{p: full resolvent order 2} to expand the resolvent to higher order, we have 
	\begin{align*}
	\frac{1}{\pi i} \int_{\Gamma^0_0} e^{\gamma^2 t} (\mcl - \gamma^2)^{-1} \gamma \, d \gamma = \frac{1}{\pi i} \int_{-a_*}^{a_*} e^{\gamma (a)^2 t} (R_0 + \gamma(a) R_1 + \mathrm{O}(a^2)) \gamma (a) \gamma'(a) \, d a. 
	\end{align*}
	The terms involving $R_0$ and $\mathrm{O}(a^2)$ decay at least as fast as $t^{-2}$, by the same arguments used in the proof of Proposition \ref{p: linear time decay}, so we focus on the term involving $R_1$. We integrate by parts to obtain 
	\begin{align*}
	\frac{1}{\pi i} \int_{-a_*}^{a_*} e^{\gamma (a)^2 t} (\gamma (a) R_1) \gamma (a) \gamma'(a) \, da &= \frac{1}{\pi i} \frac{1}{2 t} \int_{-a_*}^{a_*} (\partial_a  e^{\gamma(a)^2 t} ) (\gamma (a) R_1) \, da \\
	&= - \frac{1}{2 \pi i} \frac{1}{t}  \int_{-a_*}^{a_*} e^{\gamma (a)^2 t} \gamma'(a) R_1 \, da + \mathrm{O}(e^{- \mu t})
	\end{align*}
	for some $\mu> 0$. The boundary terms are exponentially decaying since we choose $a_*$ small enough so that $\Re \gamma (\pm a_*) < 0$. We recognize the remaining integral 
	\begin{align*}
	\int_{-a_*}^{a_*} e^{\gamma (a)^2 t} \gamma'(a) \, da 
	\end{align*}
	as a parameterization of the integral of $e^{z^2 t}$ over the contour $\Gamma^0_0$. Since $e^{z^2 t}$ is an entire function, we can deform this contour into another contour $\tilde{\Gamma}_0$ consisting of three straight line segments: one from $z = -ia_* + c_2 a_*^2$ to $z = -ia_*$, one along the imaginary axis from $z = -i a_*$ to $z = i a_*$, and one from $z = ia_*$ to $z = ia_* + c_2 a_*^2$. See the right panel of Figure \ref{fig: contours}. 

	The contributions from the lower and upper pieces of $\tilde{\Gamma}_0$ are both exponentially decaying in time, since $\Re \gamma^2$ is negative along these pieces. Hence, the dominant contribution is from the piece along the imaginary axis, and parameterizing this piece as $\gamma(a) = i a$, we have 
	\begin{align*}
	- \frac{1}{2 \pi i} \frac{1}{t}  \int_{-a_*}^{a_*} e^{\gamma (a)^2 t} \gamma'(a) R_1 \, da = - \frac{1}{2 \pi} \frac{1}{t} \int_{-a_*}^{a_*} e^{-a^2 t} \, da = - \frac{1}{2 \pi} \frac{1}{t^{3/2}} \int_{-a_* \sqrt{t}}^{a_* \sqrt{t}} e^{-w^2} \, dw. 
	\end{align*}
	The remaining integral attains its limit 
	\begin{align*}
	\int_{-a_* \sqrt{t}}^{a_* \sqrt{t}} e^{-w^2} \, dw \to \int_\R e^{-w^2} \, dw = \sqrt{\pi} 
	\end{align*}
	exponentially quickly as $t \to \infty$, so that altogether, we may write 
	\begin{align*}
	\frac{1}{\pi i} \int_{\Gamma^0_0} e^{\gamma^2 t} (\mcl - \gamma^2)^{-1} \gamma \, d \gamma = - \frac{1}{2 \pi} \frac{1}{t^{3/2}} \sqrt{\pi} R_1 + \mathrm{O}(t^{-2}),
	\end{align*}
	completing the proof of the proposition. 
\end{proof}

\section{Nonlinear stability -- proof of Theorem \ref{t: stability}}\label{s: nonlinear stability}
We write the nonlinear perturbation equation \eqref{e: v eqn} in the weighted space, by defining $p = \omega v$, from which we find
\begin{align}
p_t = \mcl p + \omega N(q_*, \omega^{-1} p), \label{e: perturbation eqn weighted}
\end{align}
where
\begin{align}
N(q_*, \omega^{-1} p) = f(q_* + \omega^{-1} p) - f(q_*) - f'(q_*) \omega^{-1} p. \label{e: nonlinearity}
\end{align}
The nonlinearity is extremely well behaved -- formally Taylor expanding, one sees 
\begin{align*}
\omega N(q_*, \omega^{-1} p) = \frac{f''(q_*)}{2} \omega^{-1} p^2 + \mathrm{O}(\omega^{-2} p^3).   
\end{align*}
In particular, the entire nonlinearity carries a factor of $\omega^{-1}$, and hence is exponentally localized, so we may use strong decay estimates on the nonlinear term in the variation of constants formula. The main difficulty has therefore already been resolved in proving sharp linear estimates in Proposition \ref{p: linear time decay}, and so we complete the proof of Theorem \ref{t: stability} in this section using a direct, classical argument, as used for instance in the proof of Theorem 1 of \cite{FayeHolzer}. 

The nonlinear equation \eqref{e: perturbation eqn weighted} is locally well-posed in $H^1_r (\R)$ for any $r \in \R$, by classical theory of semilinear parabolic equations \cite{Henry}: for initial data $p_0$ with $\|p_0\|_{H^1_r}$ sufficiently small, there exists a maximal existence time $T_* \in (0, \infty]$ and a solution $p(t)$ to \eqref{e: perturbation eqn weighted} defined up to time $T_*$, with $T_*$ depending only on $\|p_0\|_{H^1_r}$. We rewrite \eqref{e: perturbation eqn weighted} in mild form via the variation of constants formula 
\begin{align}
p(t) = e^{\mcl t} p_0 + \int_0^t e^{\mcl (t-s)} \omega N(q_*, \omega^{-1} p(s)) \, ds. 
\end{align}
Since the original nonlinearity $f$ in \eqref{e: eqn} is smooth, and $H^1 (\R)$ is a Banach algebra, it follows from Taylor's theorem that for any $s, r \in \R$, there is a nondecreasing function $K : \R_+ \to \R_+$ such that
\begin{align}
\|\omega N(q_*, \omega^{-1} p)\|_{H^1_s} \leq K(R) \|p\|_{H^1_r}^2, \label{e: nonlin quadratic}
\end{align}
if $\|\omega^{-1} p\|_{L^\infty} \leq R$. Here, the extra factor of $\omega^{-1}$ in the Taylor expansion of the nonlinearity is used to control the algebraic weights. 

We now fix $r > 3/2$ and define 
\begin{align}
\Theta (t) = \sup_{0 \leq s \leq t} (1+s)^{3/2} \|p(s)\|_{H^1_{-r}}. \label{e: theta def}
\end{align}
We prove Theorem \ref{t: stability} by obtaining global control of $\Theta$. In the proof, we will need to use the estimate 
\begin{align}
\|e^{\mcl t} p_0\|_{H^1_r} \leq C \|p_0\|_{H^1_r} \label{e: H1 boundedness}
\end{align}
for $0 < t < 1$, which holds for any fixed $r \in \R$ and follows from classical semigroup theory \cite[Section 1.4]{Henry}. 

\begin{prop}
	There exist constants $C_1, C_2 > 0$ such that the function $\Theta(t)$ from \eqref{e: theta def} satisfies
	\begin{align}
	\Theta(t) \leq C_1 \|p_0\|_{H^1_r} + C_2 K(\rho_\infty \Theta(t)) \Theta (t)^2 \label{e: theta control}
	\end{align}
	for all $t \in [0, T^*)$, where $\rho_\infty = \|\rho_r \omega^{-1}\|_{L^\infty}$. 
\end{prop}
\begin{proof}
	First assume $0 < t < 1$. Then by \eqref{e: H1 boundedness}, we have 
	\begin{align*}
	(1+t)^{3/2} \|e^{\mcl t} p_0\|_{H^1_{-r}} \leq C \|p_0\|_{H^1_{-r}} \leq C \|p_0\|_{H^1_r}.  
	\end{align*}
	For the nonlinearity, we have, again using \eqref{e: H1 boundedness} and also \eqref{e: nonlin quadratic}
	\begin{align*}
	\left| \left| \int_0^t e^{\mcl (t-s)} \omega N(q_*, \omega^{-1} p(s)) \, ds \right| \right|_{H^1_{-r}} &\leq C \int_0^t \| \omega N(q_*, \omega^{-1} p(s)) \|_{H^1_r} \, ds \\
	&\leq C \int_0^t K(\|\omega^{-1} p(s)\|_{L^\infty}) \|p(s)\|_{H^1_{-r}}^2 ds \\
	&\leq C t \sup_{0 \leq s \leq t} K(\|\omega^{-1} p(s)\|_{L^\infty}) \|p(s)\|_{H^1_{-r}}^2 \\
	&\leq C \Theta(t)^2 \sup_{0 \leq s \leq t} K(\|\omega^{-1} p(s)\|_{L^\infty}).
	\end{align*}
	Using the embedding of $H^1(\R)$ into $L^\infty(\R)$, we have 
	\begin{align*}
	C \Theta(t)^2 \sup_{0 \leq s \leq t} K\left(\|\omega^{-1} p(s)\|_{L^\infty} \right) &\leq C \Theta(t)^2 K \left( \rho_\infty \sup_{0 \leq s \leq t} \|\rho_{-r} p(s)\|_{L^\infty} \right) \\ 
	&\leq C \Theta(t)^2 K\left( \rho_\infty \sup_{0 \leq s \leq t} \|p(s)\|_{H^1_{-r}} \right) \\
	&\leq C \Theta(t)^2 K \left(\rho_\infty \Theta(t) \right). 
	\end{align*}
	Altogether, using the fact that $t \mapsto \Theta(t)$ is non-decreasing, we obtain \eqref{e: theta control} for $0 < t < 1$. 
	
	Now we let $t > 1$. For the linear evolution, we have by Proposition \ref{p: linear time decay}
	\begin{align}
	(1+t)^{3/2} \|e^{\mcl t} p_0\|_{H^1_{-r}} \leq C \frac{(1+t)^{3/2}}{t^{3/2}} \|p_0\|_{H^1_r} \leq C \|p_0\|_{H^1_r}. \label{e: t large linear estimate}
	\end{align}
	For the nonlinearity, again using Proposition \ref{p: linear time decay}, we have  
	\begin{align*}
	\|e^{\mcl(t-s)} \omega N(q_*, \omega^{-1} p)\|_{H^1_{-r}} \leq \frac{C}{(t-s)^{3/2}} \|\omega N(q_*, \omega^{-1} p)\|_{H^1_r} \, ds.
	\end{align*}
	But by \eqref{e: H1 boundedness}, we also have 
	\begin{align*}
	\|e^{\mcl (t-s)} \omega N (q_*, \omega^{-1} p)\|_{H^1_{-r}} \leq C \|\omega N(q_*, \omega^{-1} p)\|_{H^1_r} \,  
	\end{align*}
	for $(t-s) < 1$. It follows that, also using the quadratic estimate on the nonlinearity as above,
	\begin{align*}
	\left|\left| \int_0^t e^{\mcl (t-s)} \omega N(q_*, \omega^{-1} p) \, ds \right| \right|_{H^1_{-r}} &\leq C \int_0^t \frac{1}{(1+t-s)^{3/2}} \| \omega N(q_*, \omega^{-1} p)\|_{H^1_r} \, ds \\
	&\leq C K(\rho_\infty \Theta (t)) \Theta(t)^2 \int_0^t \frac{1}{(1+t-s)^{3/2}} \frac{1}{(1+s)^3} \, ds. 
	\end{align*}
	By splitting the integral into integrals from $0$ to $t/2$ and $t/2$ to $t$ and estimating each piece separately, it can be readily shown that 
	\begin{align*}
	\int_0^t \frac{1}{(1+t-s)^{3/2}} \frac{1}{(1+s)^3} \, ds \leq \frac{C}{(1+t)^{3/2}}. 
	\end{align*}
	Hence we obtain 
	\begin{align}
	(1+t)^{3/2} \left|\left| \int_0^t e^{\mcl (t-s)} \omega N(q_*, \omega^{-1} p) \, ds \right| \right|_{H^1_{-r}} \leq C K(\rho_\infty \Theta (t)) \Theta(t)^2 \label{e: t large nonlinear estimate}
	\end{align}
	for $t > 1$. Together with \eqref{e: t large linear estimate}, this shows that \eqref{e: theta control} holds for $t > 1$, completing the proof of the proposition.  
\end{proof}

\begin{proof}[Proof of Theorem \ref{t: stability}]
	Let $\|p_0\|_{H^1_r}$ be sufficiently small so that 
	\begin{equation}
	2 C_1 \|p_0\|_{H^1_r } < 1 \text{ and } 4 C_1 C_2 K(\rho_\infty) \|p_0\|_{H^1_r} < 1. \label{e: thm 1 proof smallness}
	\end{equation}
	We claim that $\Theta(t) \leq 2 C_1 \|p_0\|_{H^1_r (\R)} < 1$ for all $t \in [0, T_*)$. Since $\Theta(0) = \|p_0\|_{H^1_{-r} (\R)} \leq \|p_0\|_{H^1_r (\R)} < 2 C_1 \|p_0\|_{H^1_r (\R)}$ (choosing $C_1 > 1/2$ if necessary), continuity of $\Theta$ guarantees that $\Theta(t) < 2 C_1 \|p_0\|_{H^1_r (\R)}$ for sufficiently small $t$. Now suppose there is some time $T$ at which $\Theta (T) = 2 C_1 \|p_0\|_{H^1_r (\R)}$. Then, by \eqref{e: theta control} and the fact that $K$ is non-decreasing, we have 
	\begin{equation*}
	1 \leq 4 C_1 C_2b K(\rho_\infty) \|p_0\|_{H^1_r}, 
	\end{equation*}
	contradicting \eqref{e: thm 1 proof smallness}. Hence $\Theta(t) \leq 2 C_1 \|p_0\|_{H^1_r (\R)} < 1$ for all $t \in [0, T_*)$. In particular, we have uniform control over $\|p(t)\|_{H^1_{-r}}$, which implies that we have global existence in $H^1_{-r} (\R)$, and
	\begin{equation*}
	\|p(t)\|_{H^1_{-r}} \leq \frac{C}{(1+t)^{3/2}} \|p_0\|_{H^1_r} 
	\end{equation*}
	for all $t > 0$. This completes the proof of Theorem \ref{t: stability}, recalling that $v = \omega^{-1} p$. 
\end{proof}

\section{Asymptotics of solution profile - proof of Theorem 2}\label{s: nonlinear asymptotics}
In this section we prove Theorem \ref{t: asymptotics}, establishing an asymptotic description of the perturbation. As in the proof of Theorem \ref{t: stability}, the main difficulty has already been overcome by obtaining a detailed description of the asymptotics of the linear semiflow in Proposition \ref{p: linear time asymptotics}. We handle the nonlinearity via a direct argument, which is essentially the same as that used in \cite{GallayScheel} in the context of diffusive stability of time-periodic solutions to reaction-diffusion systems. 

We begin by decomposing the linear semigroup as 
\begin{align*}
e^{\mcl t} = \Phi^0 (t) + \Phi^\mathrm{ss}(t), 
\end{align*}
where 
\begin{align*}
\Phi^0 (t) = \frac{1}{2 \sqrt{\pi}} \frac{R_1}{t^{3/2}},
\end{align*}
and $\Phi^\mathrm{ss} (t)$ is the remainder term from Proposition \ref{p: linear time asymptotics}, which satisfies in particular 
\begin{align}
\|\Phi^\mathrm{ss} (t)\|_{H^1_r \to H^1_{-r}} \leq \frac{C}{(1+t)^2} \label{e: strong decay estimate}
\end{align}
for $t > 1$ and $r > 5/2$. We use this decomposition to rewrite the variation of constants formula as 
\begin{multline}
p(t) = \Phi^0 (t) p_0 + \Phi^\mathrm{ss} (t) p_0  + \int_0^t \Phi^0 (t-s) \omega N (q_*, \omega^{-1} p) \, ds \\ + \int_0^t \Phi^\mathrm{ss} (t-s) \omega N(q_*, \omega^{-1} p) \, ds. \label{e: voc decomp}
\end{multline}
Arguing as in the proof of Theorem \ref{t: stability}, we readily see that the parts of the solution associated with the flow under $\Phi^\mathrm{ss} (t)$ decay faster than $t^{-3/2}$, as stated in the following lemma. For the remainder of this section, we let $r > 5/2$ and assume the hypotheses of Theorem \ref{t: asymptotics} hold. 

\begin{lemma}\label{l: phi ss estimates}
	For $t > 1$, we have 
	\begin{align}
	\|\Phi^\mathrm{ss} (t) p_0\|_{H^1_{-r}} \leq \frac{C}{(1+t)^2} \|p_0\|_{H^1_{r}},
	\end{align}
	and
	\begin{align}
	\left| \left| \int_0^t \Phi^\mathrm{ss} (t-s) \omega N(q_*, \omega^{-1} p(s)) \, ds \right| \right|_{H^1_{-r}} \leq \frac{C}{(1+t)^2} \|p_0\|_{H^1_r}^2. 
	\end{align}
\end{lemma}

We now decompose the term in the nonlinearity involving $\Phi^0$ in order to identify which parts of it contribute to the leading order asymptotics and which are faster decaying. We write 
\begin{align}
\int_0^t \Phi^0 (t-s) \omega N(q_*, \omega^{-1} p(s)) \, ds = \mathcal{I}_1 (t) + \mathcal{I}_2 (t) + \mathcal{I}_3 (t) + \mathcal{I}_4 (t), \label{e: phi0 int decomposition}
\end{align}
where 
\begin{align*}
\mathcal{I}_1 (t) &= \int_{t/2}^t \Phi^0 (t-s) \omega N(q_*, \omega^{-1} p(s) ) \, ds, \\
\mathcal{I}_2 (t) &= \int_0^{t/2} (\Phi^0(t-s) - \Phi^0(t)) \omega N(q_*, \omega^{-1} p(s)) \, ds, \\
\mathcal{I}_3 (t) &= \Phi^0 (t) \int_0^\infty \omega N(q_*, \omega^{-1} p(s)) \, ds,
\end{align*}
and 
\begin{align*}
\mathcal{I}_4 (t) = - \Phi^0 (t) \int_{t/2}^\infty  \omega N(q_*, \omega^{-1} p (s)) \, ds. 
\end{align*}
\begin{lemma}\label{l: phi0 decomp}
	The terms in the decomposition \eqref{e: phi0 int decomposition} satisfy for $t > 1$
	\begin{align}
	\|\mathcal{I}_1 (t)\|_{H^1_{-r}} &\leq \frac{C}{(1+t)^3} \|p_0\|_{H^1_r}^2, \label{e: I1 estimate} \\
	\|\mathcal{I}_2 (t)\|_{H^1_{-r}} &\leq \frac{C}{(1+t)^{5/2}} \|p_0\|_{H^1_r}^2, \label{e: I2 estimate}
	\end{align}
	and
	\begin{align}
	\|\mathcal{I}_4 (t)\|_{H^1_{-r}} &\leq \frac{C}{(1+t)^{7/2}} \|p_0\|_{H^1_r}^2. \label{e: I4 estimate}
	\end{align} 
\end{lemma}
\begin{proof}
	The proofs of \eqref{e: I1 estimate} and \eqref{e: I4 estimate} proceed similarly to the proof of Theorem \ref{t: stability}, so we focus on the estimate for $\mathcal{I}_2(t)$. By the mean value theorem, we have 
	\begin{align*}
	|t^{-3/2} - (t-s)^{-3/2}| \leq C s (t-s)^{-5/2}, 
	\end{align*}
	and so it follows, using \eqref{e: nonlin quadratic} and Proposition \ref{p: linear time asymptotics}, that 
	\begin{align*}
	\|\mathcal{I}_2(t)\|_{H^1_{-r}} \leq C \|p_0\|_{H^1_r}^2 \int_0^{t/2} \frac{s}{(t-s)^{5/2}} \frac{1}{(1+s)^3} \, ds \leq \frac{C}{t^{5/2}} \|p_0\|_{H^1_r}^2 \leq \frac{C}{(1+t)^{5/2}} \|p_0\|_{H^1_r}^2
	\end{align*}
	for $t >1$, completing the proof of the lemma. 
\end{proof}

Having identified which terms are irrelevant for the leading order time dynamics, we are now ready to prove Theorem \ref{t: asymptotics}. 
\begin{proof}[Proof of Theorem \ref{t: asymptotics}]
	Using Lemmas \ref{l: phi ss estimates} and \ref{l: phi0 decomp} to separate out the faster decaying terms in the variation of constants formula \eqref{e: voc decomp}, we have 
	\begin{align}
	p(t) = \Phi^0 (t) \left( p_0 + \int_0^\infty \omega N (q_*, \omega^{-1} p(s)) \, ds \right) + \mathrm{O}(t^{-2}), \label{e: thm 2 proof leading order}
	\end{align}
	where the $\mathrm{O}(t^{-2})$ terms are understood as being controlled in $H^1_{-r}$ by $C(1+t)^{-2} \|p_0\|_{H^1_r}$ for $t$ large. By the definition of $\Phi^0$ and Lemma \ref{l: R1}, we have 
	\begin{align*}
	\Phi^0 (t) \left( p_0 + \int_0^\infty \omega N (q_*, \omega^{-1} p(s)) \, ds \right) = \alpha_* t^{-3/2} \psi, 
	\end{align*}
	where $\psi$ is the linearly growing solution to $\mcl \psi = 0$ identified in the proof of Lemma \ref{l: R1}, and $\alpha_*$ is given by 
	\begin{align}
	\alpha_* = \frac{1}{2 \sqrt{\pi}} \int_\R g^1(y) \tilde{p} (y) \, dy,
	\end{align}
	where $g^1(y)$ is the function from the expansion of the Green's function, $G^1(x,y) = \psi(x) g^1 (y)$, and 
	\begin{align}
	\tilde{p} (y) = p_0 (y) + \int_0^\infty \omega (y) N(q_*, \omega^{-1} (y) p(y, s)) \, ds. 
	\end{align}
	The asymptotic decomposition \eqref{e: thm 2 proof leading order} is therefore exactly the statement of Theorem \ref{t: asymptotics}, with this choice of $\alpha_*$. 
\end{proof}

\section{Stability at lower localization -- proofs of Theorems \ref{t: stability resolvent Holder} and \ref{t: stability resolvent blowup}}\label{s: lower localization}

We now use the ideas developed in the proof of Theorem \ref{t: stability} to understand the behavior of $e^{\mcl t}$ when acting on initial data which is less strongly localized. The nonlinearity is still strongly localized, by \eqref{e: nonlin quadratic}, so we only prove the linear estimates needed to prove Theorems \ref{t: stability resolvent Holder} and \ref{t: stability resolvent blowup}, as one may use exactly the same estimates on the nonlinearity as used in the proof of Theorem \ref{t: stability}, due to the extra exponentially decaying factor $\omega^{-1}$. 

\subsection{H\"older continuity of the resolvent -- proof of Theorem \ref{t: stability resolvent Holder}}
When acting on functions in $L^2_r (\R)$ for $\frac{1}{2} < r < \frac{3}{2}$, the resolvent $(\mcl - \gamma^2)^{-1}$ is no longer Lipschitz in $\gamma$, but instead has some H\"older continuity. We exploit this H\"older continuity to obtain sharp time decay rates exactly as in the proof of the $t^{-3/2}$ decay for $r > 3/2$ in Proposition \ref{p: linear time decay}. 
\begin{prop}\label{p: resolvent Holder}
	Let $\frac{1}{2} < r < \frac{3}{2}$, $s < r - 2$, and fix some $\alpha$ with $0 < \alpha < r - \frac{3}{2} + \min \left( 1, -\frac{1}{2} - s \right)$. Then 
	\begin{align}
	(\mcl-\gamma^2)^{-1} = R_0 + \mathrm{O}(|\gamma|^\alpha) \label{e: resolvent holder}
	\end{align}
	in $\mathcal{B} (L^2_r (\R), H^{2m-1}_s (\R))$ for $\gamma$ small with $\gamma^2$ to the right of the essential spectrum of $\mcl$. 
\end{prop}

Using the far-field/core decomposition argument of Section \ref{ss: farfield core leading order}, the proof of this proposition reduces to obtaining the corresponding estimate for the asymptotic resolvent $(\mcl_+ - \gamma^2)^{-1}$, acting on odd functions. This follows from explicit estimates on the resolvent kernel $G_\gamma^+$. As in Section \ref{sec: asymptotic resolvents}, we decompose $G^+_\gamma$ as 
\begin{align*}
G^+_\gamma = G^\mathrm{heat}_\gamma + (G^c_\gamma - G^\mathrm{heat}_\gamma) + \tilde{G}^c_\gamma + G^h_\gamma.
\end{align*}
The worst behaved pieces are $G^\mathrm{heat}_\gamma$ and $G^c_\gamma - G^\mathrm{heat}_\gamma$. We use the fact that we are acting on odd data only to replace convolution with $G^\mathrm{heat}_\gamma$ with integration against $G^\mathrm{odd}_\gamma (x,y)$ defined in \eqref{e: G odd}. Using similar methods as in Section \ref{sec: asymptotic resolvents}, we obtain the following estimates on the parts of the resolvent kernel. We also make use of the fact that for $\beta > 0$, $\langle x \rangle^\beta \langle y \rangle^{-\beta} \leq \langle x - y \rangle^\beta$.
\begin{lemma}
	For $1 > a > \alpha > 0$, the integral kernels $G^\mathrm{odd}_\gamma$, $G^c_\gamma - G^\mathrm{heat}_\gamma$, and $\tilde{G}^c_\gamma$ satisfy the following estimates for $\gamma$ small with $\gamma^2$ to the right of the essential spectrum of $\mcl$,
	\begin{align}
	|G^\mathrm{odd}_\gamma (x,y) - 2 \nu_0 \min (x,y)| &\leq C |\gamma|^\alpha \langle x \rangle^a \langle y \rangle^{1+\alpha-a}, \\
	|G^c_\gamma(x-y) - G^\mathrm{heat}_\gamma (x-y)| &\leq C |\gamma|^\alpha \langle x \rangle^a \langle y \rangle^{1 + \alpha - a}, 
	\end{align}
	and 
	\begin{align}
	|\tilde{G}^c_\gamma (x-y) - \tilde{G}^c_0 (x-y)| \leq C |\gamma|^\alpha |x-y|^\alpha. 
	\end{align}
\end{lemma}
Together with the fact that convolution with $G^h_\gamma$ is analytic in $\gamma^2$ as an operator on $L^2 (\R)$, we obtain
\begin{align*}
(\mcl_+ - \gamma^2)^{-1} = R_0 + \mathrm{O}(|\gamma|^\alpha)
\end{align*}
in $\mathcal{B}(L^2_{r, r} (\R), H^{2m-1}_{s, s} (\R))$ for the values of $r, s, \alpha$ and $\gamma$ specified in Proposition \ref{p: resolvent Holder}. Using this and repeating the far-field/core decomposition argument in Section \ref{ss: farfield core leading order}, we obtain Proposition \ref{p: resolvent Holder}. We use this regularity of the resolvent to prove the following time decay estimate for the semigroup. 

\begin{prop}\label{p: Holder time decay}
	Let $\frac{1}{2} < r < \frac{3}{2}$ and $s < r-2$. For any $\alpha$ with $0 < \alpha < r - \frac{3}{2} + \min \left( 1, -\frac{1}{2} - s \right)$, there is a constant $C > 0$ such that the semigroup $e^{\mcl t}$ satisfies for $t > 0$ 
	\begin{align}
	\|e^{\mcl t}\|_{L^2_r \to H^{2m-1}_{s}} \leq \frac{C}{t^{1+\frac{\alpha}{2}}}. 
	\end{align}
\end{prop}
\begin{proof}
	We use the same contours as in the proof of Proposition \ref{p: linear time decay}, pictured in Figure \ref{fig: contours}. We follow the proof of this proposition -- again, the relevant part of the contour is the piece $\Gamma^0_0$ which touches the origin. We use Proposition \ref{p: resolvent Holder} to write 
	\begin{align*}
	\frac{1}{\pi i} \int_{\Gamma^0_0} e^{\gamma^2 t} (\mcl - \gamma^2)^{-1} \gamma \, d \gamma = \frac{1}{\pi i} \int_{\Gamma^0_0} e^{\gamma^2 t} (R_0 + \mathrm{O}(|\gamma|^\alpha)) \gamma \, d \gamma. 
	\end{align*}
	As in the proof of Proposition \ref{p: linear time decay}, we see that the integral associated to $R_0$ decays exponentially in time, and the remainder can be estimated by parametrizing the contour with $\gamma(a) = i a + c_2 a^2$ and changing variables to $z \sim a \sqrt{t}$, which readily gives 
	\begin{align*}
	\left| \left| \frac{1}{\pi i} \int_{\Gamma^0_0} e^{\gamma^2 t} (R_0 + \mathrm{O}(|\gamma|^\alpha)) \gamma \, d \gamma \right| \right|_{L^2_r \to H^{2m-1}_s} \leq \frac{C}{t^{1+\frac{\alpha}{2}}},
	\end{align*}
	as desired. 
\end{proof}

Theorem \ref{t: stability resolvent Holder} follows from applying Proposition \ref{p: Holder time decay} in a direct nonlinear stability argument as in Section \ref{s: nonlinear stability}. 

\subsection{Blowup of the resolvent -- proof of Theorem \ref{t: stability resolvent blowup}}
The resolvent $(\mcl - \gamma^2)^{-1}$ acting on $L^2_r (\R)$ for $r < 1/2$ is no longer uniformly bounded for $\gamma$ small with $\gamma^2$ to the right of the essential spectrum. However, by again explicitly analyzing the asymptotic operators and transferring these estimates to the full resolvent with a far-field/core decomposition, we can quantify the blowup of the resolvent and thereby obtain decay rates for the semigroup. The key result is the following blowup estimate. 

\begin{prop}\label{p: resolvent blowup}
	Let $-\frac{3}{2} < r < \frac{1}{2}$ and $s < r -2$. For any $\beta$ with $\frac{1}{2} - r < \beta < - s - \frac{3}{2}$, there is a constant $C > 0$ such that  
	\begin{align}
	\|(\mcl - \gamma^2)^{-1}\|_{L^2_r \to H^{2m-1}_s} \leq \frac{C}{|\gamma|^\beta}
	\end{align}
	for $\gamma$ small with $\mathrm{Re } \, \gamma \geq \frac{1}{2} | \mathrm{Im } \, \gamma|$. 
\end{prop}

As in the previous sections, we start by proving the corresponding result for the asymptotic operator $(\mcl_+ - \gamma^2)$. This estimate follows from the explicit estimates on the resolvent kernel that we collect in the following lemma. 

\begin{lemma}\label{l: blowup resolvent pointwise estimates}
	For any $\beta > 0$, the integral kernels $G^\mathrm{odd}$, $G^c_\gamma - G^\mathrm{heat}_\gamma$, and $\tilde{G}^c_\gamma$ satisfy the following estimates for $\gamma$ small with $\mathrm{Re } \, \gamma \geq \frac{1}{2} |\mathrm{Im } \, \gamma|$
	\begin{align*}
	|G^\mathrm{odd}_\gamma (x,y)| &\leq \frac{C}{|\gamma|^\beta} \langle x \rangle^{\beta + 1} \langle y \rangle^{-\beta}, \\
	|G^c_\gamma (x-y) - G^\mathrm{heat}_\gamma (x-y)| &\leq \frac{C}{|\gamma|^\beta} \langle x \rangle^\beta \langle y \rangle^{-\beta},
	\end{align*}
	and 
	\begin{align*}
	|\tilde{G}^c_\gamma (x-y) | \leq \frac{C}{|\gamma|^\beta} \langle x \rangle^\beta \langle y \rangle^{-\beta}. 
	\end{align*}
\end{lemma}

\begin{figure}
	\centering
	\includegraphics[width = 0.8\textwidth]{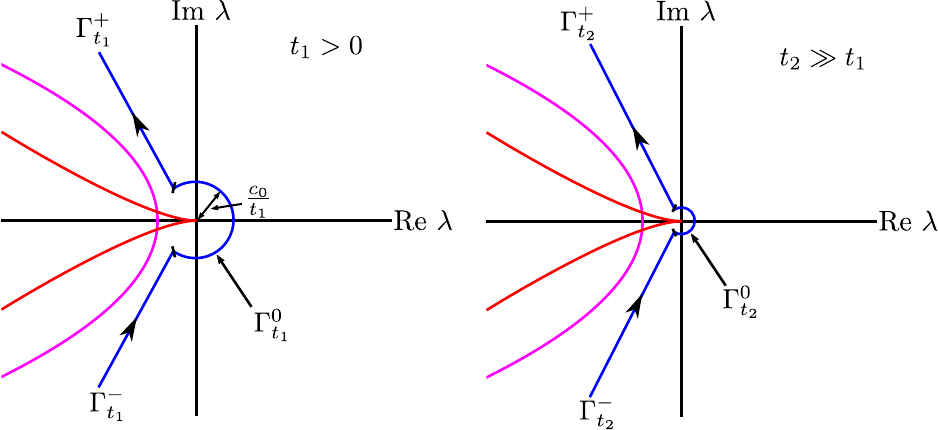}
	\caption{Fredholm borders of $\mcl$ (red, magenta) together with the integration contour used in the proof of Proposition \ref{p: linear time decay blowup}, for $t = t_1 > 0$ (left) and for $t_2 \gg t_1$ (right).}
	\label{fig: keyhole contour}
\end{figure}

$G^h_\gamma$ is uniformly exponentially localized in space for $\gamma$ small, and so convolution with $G^h_\gamma$ is uniformly bounded in $\gamma$ for $\gamma$ small between any two algebraically weighted spaces. From this and Lemma \ref{l: blowup resolvent pointwise estimates}, we obtain 
\begin{align*}
\|(\mcl_+ - \gamma^2)^{-1}\|_{L^2_{r,r} \to H^{2m-1}_{s,s}} \leq \frac{C}{|\gamma|^\beta}
\end{align*}
for $r, s, \beta$, and $\gamma$ as in Proposition \ref{p: resolvent blowup}. Again, using the far-field/core decomposition in Section \ref{ss: farfield core leading order}, we readily obtain Proposition \ref{p: resolvent blowup} from this estimate. 

We now use this control of the blowup of the resolvent to obtain time decay estimates for the semigroup. Since the resolvent is blowing up at the origin, we can no longer shift our integration contour all the way to the essential spectrum. Instead, we use a classical semigroup theory argument, integrating along a circular arc as pictured in Figure \ref{fig: keyhole contour}.

\begin{prop}\label{p: linear time decay blowup}
	Let $-\frac{3}{2} < r < \frac{1}{2}$ and $s < r-2$. For any $\beta$ with $\frac{1}{2} - r < \beta < -s - \frac{3}{2}$, there is a constant $C > 0$ such that the semigroup $e^{\mcl t}$ satisfies for $t > 1$
	\begin{align}
	\|e^{\mcl t} \|_{L^2_r \to H^{2m-1}_s} \leq \frac{C}{t^{1-\frac{\beta}{2}}}. 
	\end{align}
\end{prop}

\begin{proof}
	We integrate over the contour $\Gamma_t = \Gamma^-_t \cup \Gamma^0_t \cup \Gamma^+_t$ pictured in Figure \ref{fig: keyhole contour}. The important piece is the circular arc $\Gamma^0_t$, which we parameterize for $t > 1$, fixed, as 
	\begin{align*}
	\Gamma^0_t = \left\{ \lambda(\phi) = \frac{c_0}{t} e^{i \phi} : \phi \in (- \phi_0, \phi_0) \right\}
	\end{align*}
	with $c_0$, and $\phi_0$ chosen appropriately so that $\Gamma^0_t$ does not intersect the essential spectrum of $\mcl$ for $t > 1$, and so that Proposition \ref{p: resolvent blowup} holds for $\gamma^2 \in \Gamma^0_t$ for $t$ sufficiently large. The contours $\Gamma^\pm_t$ are rays connecting the $\Gamma^0_t$ to infinity, in the left half plane, as pictured. The semigroup $e^{\mcl t}$ may be written as 
	\begin{align*}
	e^{\mcl t} = -\frac{1}{2\pi i} \int_{\Gamma_t} e^{\lambda t} (\mcl - \lambda)^{-1} \, d \lambda. 
	\end{align*}
	The contributions to this integral from $\Gamma^\pm_t$ are exponentially decaying in time, so we focus only on the integral over $\Gamma^0_t$. Here we change variables to $\xi = \lambda t$, so that 
	\begin{align*}
	\frac{1}{2 \pi i} \int_{\Gamma^0_t} e^{\lambda t} (\mcl - \lambda)^{-1} \, d \lambda = \frac{1}{2\pi i} \frac{1}{t} \int_{\Gamma^0_1} e^{\xi} \left( \mcl - \frac{\xi}{t} \right)^{-1} \, d \xi. 
	\end{align*}
	By Proposition \ref{p: resolvent blowup}, we have for $t$ large
	\begin{align*}
	\left|\left| \left( \mcl - \frac{\xi}{t} \right)^{-1} \right| \right|_{L^2_r \to H^{2m-1}_s} \leq C \frac{t^{\beta/2}}{|\xi|^{\beta/2}},  
	\end{align*}
	and so 
	\begin{align*}
	\left| \left| \frac{1}{2 \pi i} \int_{\Gamma^0_t} e^{\lambda t} (\mcl - \lambda)^{-1} \, d \lambda \right| \right|_{L^2_r \to H^{2m-1}_s} \leq \frac{C}{t^{1-\frac{\beta}{2}}} \int_{\Gamma^0_1} |e^\xi| \frac{1}{|\xi|^{\beta/2}} \, d \xi \leq \frac{C}{t^{1-\frac{\beta}{2}}},
	\end{align*}
	as desired. 
\end{proof}

Theorem \ref{t: stability resolvent blowup} readily follows from Proposition \ref{p: linear time decay blowup} and a direct nonlinear stability argument as in Section \ref{s: nonlinear stability}. Again, we emphasize that the nonlinearity is still exponentially localized due to the extra factor of $\omega^{-1}$, and so we may use strong decay estimates on the nonlinearity to close this argument. 

\section{Examples and discussion}\label{s: discussion} \hfill

\noindent \textbf{Second order equations.} The classical setting for studying invasion fronts is that of second order scalar parabolic equations
\begin{align}
u_t = u_{xx} + f(u).
\end{align}
It is well known that if, for instance, $f(0) = f(1) = 0$, $f'(0) > 0$, $f'(1) < 0$, and $f''(u) <0$, then there exist monotone traveling fronts in this equation for all speeds $c \geq c_\mathrm{lin} = 2 \sqrt{f'(0)}$, and that the linearization about the critical front, with $c = c_\mathrm{lin}$, satisfies our spectral assumptions. In this case unstable point spectrum is ruled out using Sturm-Liouville type arguments \cite[Theorem 5.5]{Sattinger}. A more detailed discussion of conditions on $f$ which guarantee the existence of monotone fronts above certain speed thresholds is given in \cite{HadelerRothe}. 

To put our spectral assumptions in the context of dichotomies between pushed and pulled fronts, we consider a bistable nonlinearity with a parameter $0 < \mu < \frac{1}{2}$
\begin{align}
u_t = u_{xx} + u(u+\mu)(1-\mu - u). 
\end{align}
This equations has three spatially uniform equilibria, of which $u \equiv 1 - \mu$ and $u \equiv -\mu$ are stable, while $u \equiv 0$ is unstable. It is shown in \cite{HadelerRothe} that if $\frac{1}{3} < \mu \leq \frac{1}{2}$, then there exist monotone fronts connecting $1-\mu$ at $-\infty$ to $0$ at $+\infty$ for all speeds $c \geq c_\mathrm{lin} = 2 \sqrt{\mu(1-\mu)}$ --- the fronts are pulled, in the sense that the minimal propagation speed matches the linear spreading speed. In this case, our results apply to the critical front with $c = c_\mathrm{lin}$ (one may rescale the amplitude of $u$ by $(1-\mu)^{-1}$ to scale the stable state on the left to $u \equiv 1$, if desired). 

However, if $0 < \mu < \frac{1}{3}$, then there exist monotone fronts connecting $1-\mu$ to $0$ only for $c \geq c_{\mathrm{min}} = \frac{1+\mu}{\sqrt{2}} > c_\mathrm{lin}$ -- the fronts are \textit{pushed}, in that the minimal propagation speed is greater than the linear spreading speed, due to amplifying effects of the nonlinearity. In this case, there still exists a front with $c = c_\mathrm{lin}$, but this front is not monotone, and hence its linearization has an unstable eigenvalue by Sturm-Liouville considerations, and our assumption on spectral stability, Hypothesis \ref{hyp: resonances}, no longer applies. Since this front is unstable, the relevant question for the dynamics of this system is the stability of the pushed front, with $c = c_{\mathrm{min}}$. This is more straightforward than the stability of the pulled fronts considered here, as the essential spectrum can be stabilized with exponential weights, leaving only a translational eigenvalue at the origin. One then obtains orbital stability of the pushed front by projecting away the effect of this translational eigenvalue, with exponential in time decay to a translate of the front \cite{SattingerWeightedNorms}. 

At the transition between pushed and pulled fronts, $\mu = \frac{1}{3}$, we have $c_{\mathrm{min}} = c_{\mathrm{lin}}$, and there is a monotone front connecting $1 - \mu$ to $0$ with this speed. This front is marginally spectrally stable, satisfying Hypotheses \ref{hyp: spreading speed} and \ref{hyp: left} with no unstable point spectrum. However, in this case the front has strong exponential decay, $q_*(x) \sim e^{-\eta_* x}$ as $x \to \infty$, and so its derivative contributes to a resonance of the linearization in the appropriate exponentially weighted space. Hence our analysis does not apply to this threshold case, and to our knowledge, precise decay rates for perturbations to the front have not been identified. 

\noindent \textbf{The extended Fisher-KPP equation.} The extended Fisher-KPP equation 
\begin{align}
u_t = - \eps^2 u_{xxxx} + u_{xx} + f(u) 
\end{align}
may be derived from reaction-diffusion systems as an amplitude equation near certain co-dimension 2 bifurcation points \cite{RottschaferDoelman}. If $f$ is of Fisher-KPP type, e.g. $f(1) = f(0) = 0, f'(0) > 0, f'(1) < 0$, and $f''(u) < 0$ for $u \in (0,1)$, then this equation is a singular perturbation of the Fisher-KPP equation, and using methods of geometric singular perturbation theory, Rottsch\"afer and Wayne established in \cite{RottschaferWayne} that, exactly as for the Fisher-KPP equation, there is a linear spreading speed $c_\mathrm{lin}(\eps)$ such that for all speeds $c \geq c_\mathrm{lin} (\eps)$, there exist monotone front solutions connecting $1$ at $-\infty$ to $0$ at $+\infty$. In the same paper, Rottsch\"afer and Wayne also considered stability of these fronts using energy methods, establishing asymptotic stability but without identifying the temporal decay rate. 

Using functional analytic methods developed to study bifurcation of eigenvalues near resonances in the essential spectrum \cite{PoganScheel} and to regularize singular perturbations \cite{GohScheel}, one can view the analysis of the linearization about the critical front here as a perturbation of the corresponding problem for the underlying Fisher-KPP equation, and thereby show that for $\eps$ small the linearization has no unstable point spectrum and no resonance at the origin \cite{AveryGarenaux}. Our results therefore apply in this case, extending the stability results of \cite{RottschaferWayne} by giving a precise description of decay rates for perturbations. We emphasize that here stability cannot be proven using comparison principles.  

\noindent \textbf{Systems of equations.} Our approach can be readily adapted to systems of parabolic equations satisfying our assumptions. A version of Theorem \ref{t: stability} was recently proved for pulled fronts in a diffusive Lotka-Volterra model by Faye and Holzer \cite{FayeHolzerLotkaVolterra}, using the competitive structure of the system to exclude unstable eigenvalues with the comparison principle. Using our methods, one should obtain an extension of this result, removing the requirement for localization of perturbations on the left, as well as versions of Theorems \ref{t: asymptotics} through \ref{t: stability resolvent blowup} in this setting. 

Our next two examples highlight the importance of our assumption that the linearization about the front is marginally spectrally stable in a fixed exponential weight, with a focus on how this assumption relates to ensuring that the linear spreading speed identified in Hypothesis \ref{hyp: spreading speed} is the selected nonlinear propagation speed. The first example gives a system in which this assumption on exponential weights is both necessary and sufficient for nonlinear propagation at the linear spreading speed. Consider the following system of equations
\begin{align*}
u_t &= u_{xx} + u-u^3 + \eps v \\
v_t &= dv_{xx}  + g(v),
\end{align*}
with $d > 0$, $g(0) = 0$, and $g'(0) < 0$. This system has a front solution $(u(x,t), v(x,t)) = (q_* (x-2t), 0)$, where $q_*$ is the critical Fisher-KPP front in the first equation, with $q_*(-\infty) = 1$ and $q_*(\infty) = 0$. The linearized equations about $(u, v) = (0,0)$, in the co-moving frame with speed $2$, are 
\begin{align*}
u_t &= u_{xx} +2 u_x + u, \\
v_t &= d v_{xx} + 2 v_x + g'(0) v.  
\end{align*}
In order to stabilize the essential spectrum in the first equation, we use a smooth positive exponential weight 
\begin{align*}
\omega(x) = \begin{cases}
e^x, & x \geq 1, \\
1, & x \leq -1, 
\end{cases}
\end{align*}
writing $U = \omega u, V = \omega v$. The linearized equations for $U$ and $V$ about $U = V = 0$ for $x > 1$ are then
\begin{align*}
U_t &= U_{xx}, \\
V_t &= d V_{xx} + (2 - 2d) V_x + (d - 2 + g'(0) ) V. 
\end{align*}
In order to have marginal spectral stability in a fixed exponential weight, as required by Hypothesis \ref{hyp: resonances}, we must have $d < 2 - g'(0)$. Holzer demonstrated in \cite{HolzerAnomalous} that if this condition is violated, then the system exhibits anomalous spreading --- the nonlinear propagation is no longer determined by the condition in Hypothesis \ref{hyp: spreading speed}. In this case, the assumption of marginal stability in a fixed exponential weight, which we use in our analysis, is necessary and sufficient for nonlinear invasion at the linear spreading speed. Our results should apply in this system for $d < 2 - g'(0)$, using smallness of the coupling coefficient $\eps$ to obtain the spectral stability in Hypothesis \ref{hyp: resonances} via a perturbative argument. 

If one modifies this system slightly, the situation becomes more subtle. The key modification is to replace the linear coupling term $\eps v$ with quadratic coupling, as considered by Faye et al. in \cite{FayeHolzerScheel}. The examples there are amplitude equations which can be derived from systems in which a homogeneous state undergoes a pitchfork bifurcation simultaneously with a Turing bifurcation, and have the form
\begin{align*}
u_t &= u_{xx} + u - u^3 + a_1 v^2 + a_2 u v^2, \\
v_t &= d v_{xx} - b_1 v - b_2 v^3. 
\end{align*}
Such systems can be derived as amplitude equations from the class of scalar parabolic equations we consider here, if $f(u) = \mu u - u^3$ and $\mathcal{P}$ is an 8th order even polynomial satisfying 
\begin{align*}
\mathcal{P}(0) &= \eps^2, \, \, \, \, \qquad \mathcal{P}'(0) = 0, \, \, \, \, \, \mathcal{P}''(0) = 2, \\
\mathcal{P}(0) &= - b_1 \eps^2, \quad \mathcal{P}'(i) = 0, \quad \mathcal{P}''(i) = 2d. 
\end{align*}
The linearization about the unstable state $(u, v) = (0, 0)$ is unchanged from the previous example, and so $d < 2 - b_1$ is still a necessary condition for the linearization to have marginally stable essential spectrum in a fixed exponential weight. However, because the coupling terms are all at least quadratic in $v$, unlike in the previous example the linearization about the unstable state is still marginally \textit{pointwise} stable at $c = 2$ even for $d \gtrsim 2 - b_1$, in the sense that solutions to 
\begin{align*}
u_t &= u_{xx} + c u_x + u, \\
v_t &= d v_{xx} + c v_x - b_1 v
\end{align*}
with compactly supported initial data decay exponentially to zero, uniformly in space, for $c > 2$, but grow for $c < 2$ \cite{HolzerScheelPointwiseGrowth}. Hence, if $d$ is only slightly larger than $2 - b_1$, the linear spreading speed is still $c = 2$, and Faye et al. show using pointwise semigroup methods \cite{HolzerFayeScheelSiemer} that the pulled front traveling with this speed is nonlinearly stable. Hence this example demonstrates that marginal stability in a fixed exponentially weighted space is \textit{not} necessary for invasion at the linear spreading speed, although we have used this assumption for our analysis here. For large values of $d$ in this system, the coupling does change the spreading speed to a ``resonant spreading speed'' which is still linearly determined but not by a simple pinched double root criterion as in Hypothesis \ref{hyp: spreading speed} \cite{FayeHolzerScheel}. 

\bibliographystyle{plain}
\bibliography{references}

\end{document}